\documentclass[oneside,a4paper,english, 11pt]{amsart}

\input{artmymacros.sty}

\title{Functions on Irreducible Components of the Emerton-Gee Stack}
\author{Eivind Otto Hjelle}
\address{Department of Mathematics,
Northwestern University, 
2033 Sheridan Road, 
Evanston, Illinois 60208, USA}
\email{eohjelle@gmail.com}

\author{Louis Jaburi}
\address{Department of Mathematics,
Imperial College London, 
London SW7 2AZ,
UK}
\email{l.jaburi20@imperial.ac.uk}

\author{Rachel Knak}
\address{Department of Mathematics,
University of Arizona,
617 N Santa Rita Ave, 
Tucson, AZ 85721, USA}
\email{rachelknak@math.arizona.edu}

\author{Hao Lee}
\address{Department of Mathematics,
University of Chicago, 
5734 South University Avenue, 
Chicago, Illinois 60637, USA}
\email{haolee@math.uchicago.edu}

\author{Wang Shenrong}
\address{Department of Mathematics,
Johns Hopkins University,
3400 North Charles Street,
Baltimore, Maryland, 21210, USA}
\email{swang219@jhu.edu}

\begin{document}

\begin{abstract}
    Let $K / \mathbb{Q}_p$ be a finite unramified extension, with $k$ being the residue field of its ring of integers. Let $\mathcal{X}_n$ denote the Emerton-Gee stack parametrizing \'etale $(\varphi,\Gamma_K)$-modules of rank $n$.
    It is known since the work of Emerton-Gee \cite{emerton-gee} that the irreducible components of the reduced special fiber of $\mathcal{X}_n$ are labeled by Serre weights $\sigma$ of $\GL_n(k)$.
    If such a component is denoted $\mathcal{X}(\sigma)$, we prove that $\mathcal{O}(\mathcal{X}(\sigma)) \cong \mathbb{F}[x_1,x_2,\dots,x_{n-1},x_n^{\pm 1}]$ when $\sigma$ is sufficiently generic.
\end{abstract}

\maketitle

\tableofcontents

\section{Introduction}

The aim of this paper is to compute the global sections of certain  irreducible components of the reduced Emerton-Gee stack. 
Recall that given an unramified extension $K = \Q_{p^f}$ of $\Q_p$ with residue field $k=\bF_{p^f}$ and a sufficiently large finite extension $E/\Q_p$, with ring of integers $\cO$, uniformizer $\varpi$ and residue field $\bF$, we can construct the Emerton-Gee stack $\cX_n$ over $\Spf(\cO)$ parameterizing \'etale $(\varphi,\Gamma_K)$-modules of rank $n$ \cite{emerton-gee}. 
The relation between  \'etale $(\varphi,\Gamma_K)$-modules and representations of the local Galois group $G_K$ tells us that the $\cO$ (resp. $\bF$)-points of $\cX_n$ is the groupoid of Galois representations $\{\rho: G_K \to \GL_n(\cO)\}$ (resp.  the groupoid of Galois representations $\{\bar \rho: G_K \to \GL_n(\bF)\}$). 

According to \cite[Theorem 5.5.12]{emerton-gee} the reduced substack $\cX_{n,\text{red}}/\bF$ admits a decomposition 
$$\cX_{n,\red,\bF}=  \bigcup_{\sigma} \cX(\sigma) $$
into irreducible components, which are indexed by the Serre weights $\sigma$ of $\GL_n(\cO_K)$. Our main theorem is as follows:
\begin{repthm}{main-theorem}
If the Serre weight $\sigma$ is sufficiently generic, then
\begin{equation}
\label{main-equation}
\cO(\cX(\sigma))\cong \F[x_1,x_2,\dots,x_{n-1},x_n^{\pm 1}] .
\end{equation}
\end{repthm}
\begin{rmk}
    More precisely, the theorem applies for $\sigma$ which is $(3n-1)$-deep (defined in Section \ref{sec:lowest-alcove-presentations}).
    This is when Theorem \ref{presentation-of-EG-component} applies.
\end{rmk}

Let us provide some context for this theorem.
A conjecture of Emerton-Gee-Hellmann \cite{emerton-gee-hellmann} predicts the existence of a fully faithful functor of the form
\begin{align*}
       \mathfrak{A} : \begin{Bmatrix}
     \textrm{Some category of smooth } \\\GL_n(K)\textrm{-representations}
    \end{Bmatrix} \longrightarrow
    \begin{Bmatrix} \textrm{Some category of quasi-coherent} \\ \textrm{sheaves on } \cX_n \end{Bmatrix}
\end{align*}
where we intentionally remain vague on the precise formulation. The former category should contain $\cInd^{\GL_n(K)}_{\GL_n(\cO_K)}(\sigma)$ as an object, corresponding to a sheaf $M_\sigma = \mathfrak{A}(\cInd^{\GL_n(K)}_{\GL_n(\cO_K)}(\sigma))$ on the other side. 
For $\sigma$ in the Fontaine-Laffaille range, the sheaf $M_\sigma$ is expected to be a line bundle supported on $\cX(\sigma)$, in which case the functor $\mathfrak{A}$ produces an isomorphism
$$ \End(\cInd_{\GL_n(\cO_K)}^{\GL_n(K)} \sigma) \cong \End(M_\sigma) \cong \cO(\cX(\sigma)) . $$
We warn the reader that $M_\sigma$ is not expected to be a line bundle in general.
Nevertheless, Theorem \ref{main-theorem} shows that there is still such an isomorphism.

A proof of Theorem \ref{main-theorem} for the special case $n=2$ can be found in \cite[Section 7.6.4]{emerton-gee-hellmann}.
Their proof relies on a description of $\cX(\sigma)$ as a quotient stack\footnote{This presentation as a quotient stack is actually for the version with fixed determinant.}
$$ \cX(\sigma) \cong \left[ \prod_{i=0}^{f-1} (\bA^2 - \lbrace (0,0) \rbrace) / \bG_m^f \right] , $$
with respect to a certain action of $\bG_m^f$ by shifted conjugation.
Using this description, the ring of global functions can be computed ``by hand''.
In the more general setting of our theorem, we use a similar description of $\cX(\sigma)$ coming from \cite{localmodels}.

Let $T$ be the diagonal torus of $\GL_n$. We also prove a result about functions on certain $T$-torsors of affine Schubert varieties that may be of independent interest.
This is Theorem \ref{important}, which is crucially used in our proof of Theorem \ref{main-theorem}.

\subsection{Interpretation of results}\label{section:interpretation}
The Satake isomorphism in characteristic $p$ \cite{herzig-satake} shows that the Hecke algebra 
$$\cH(\sigma) \defeq \End_{\GL_n(K)}(\cInd^{\GL_n(K)}_{\GL_n(\cO_K)}\sigma) \cong \F[T_1,T_2,\dots,T_{n-1},T_n^{\pm 1}]$$ 
is isomorphic to a polynomial ring as in Equation (\ref{main-equation}). 
Therefore, Theorem \ref{main-theorem} as stated shows that there exists an isomorphism 
\begin{equation}
    \label{eq:hecke-iso}
    \cO(\cX(\sigma)) \cong \cH(\sigma) . 
\end{equation}
This is just an abstract isomorphism of polynomial algebras, and it is not at all obvious how to construct a ``meaningful'' map in either direction.

Let us briefly sketch how it should be possible to characterize the functions $x_i \in \cO(\cX(\sigma))$ as certain renormalized Hecke operators on the ordinary locus of $\cX(\sigma)$; additional details will be provided in Section \ref{section:hecke-operators}.
The irreducible component $\cX(\sigma)$ embeds into the special fiber of $\cX^{\eta,\tau}$ for a suitable choice of inertial type $\tau$, and the rigid generic fiber admits a map to a the analytification of a moduli stack $\operatorname{WD}_\tau$ of Weil-Deligne representations.
\[\begin{tikzcd}
	{\mathcal{X}^{\eta,\tau,\text{rig}}_E} & {\mathcal{X}^{\eta,\tau}} & {\mathcal{X}^{\eta,\tau}_{\text{red},\mathbb{F}}} & {\mathcal{X}(\sigma)} \\
	{\operatorname{WD}_\tau^{\text{an}}}
	\arrow[hook', from=1-3, to=1-2]
	\arrow[hook, from=1-1, to=1-2]
	\arrow[hook', from=1-4, to=1-3]
	\arrow[from=1-1, to=2-1]
\end{tikzcd}\]
The upshot is that the Hecke algebra $\cH(\sigma(\tau))$ describe functions on the Weil-Deligne stack $\operatorname{WD}_\tau$, which therefore pull back to functions on $\cX^{\eta,\tau,\text{rig}}_E$.
Suitable renormalizations of the standard operators $T_i \in \cH(\sigma(\tau))$ extend to functions on a substack of $\cX^{\eta,\tau}$, and by restriction to the special fiber these agree with the $x_i$'s on the ordinary locus of $\cX(\sigma)$.

\subsection{Outline of proof}
Let us explain in more detail the idea of the proof and the reason for the genericity assumption. Assume for simplicity that $K= \Q_p$ and thus $f=1$.

As we will recall in more detail in Section \ref{section:step0}, \cite[Theorem 7.4.2]{localmodels} implies, under the conditions in loc. cit., that $\cX(\sigma)$ admits a surjection 
$$[\widetilde{\cC}_\sigma / T] \twoheadrightarrow\cX(\sigma)$$ 
where $\widetilde{\cC}_\sigma$ is certain closed subspace of $\widetilde{\Fl}= I_1\setminus LG$, a $T$-torsor of the affine flag variety $\Fl$. 
Now, again by \cite[Theorem 7.4.2]{localmodels}, if the genericity assumption is fulfilled, then the above surjection is an isomorphism. 
Therefore
$$ \cO(\cX(\sigma))\cong \cO(\widetilde{C}_\sigma)^{T} $$
and it is the latter object that we will be computing.
This computation relies on various open charts of $\tld{\cC}_\sigma$ that will be introduced in Section \ref{section:open-charts}.
The identification of $\cO(\tld{\cC}_\sigma)^T$ proceeds in three steps:
\begin{enumerate}
    \item We first exhibit an open $\tld{U}_{s_w} \hookrightarrow \widetilde{\cC}_\sigma$ and compute $\cO(\tld{U}_{s_w})^{T}\cong \bF[x_1^{\pm 1},...,x_n^{\pm 1}]$. We will also compute $T$-invariant sections on some closely related auxiliary charts. This is done in sections \ref{Simple upper bound} and \ref{section:auxiliary-charts}.
    \item We consider a larger open $\tld{U}_{s_w} \hookrightarrow \tld{U} \hookrightarrow \widetilde{\cC}_\sigma$ and compute via the first step that $\cO(\tld{U})^T\hookrightarrow\cO(\tld{U}_{s_w})^{T}$ is precisely $\bF[x_1,..., x_{n-1}, x_n^{\pm 1}]\hookrightarrow \bF[x_1^{\pm 1},...,x_n^{\pm 1}]$. This is Section \ref{section:step2}. Note that we will exhibit the functions $x_i$ as certain minors here.
    \item By the last step we have $\cO(\widetilde{\cC}_\sigma)^{T} \hookrightarrow \cO(\tld{U})^{T}\cong \bF[x_1,..., x_{n-1}, x_n^{\pm 1}]$ and we have to show that this injection is in fact an isomorphism, i.e. we show that each $T$-invariant function on $\tld{U}$ can be extended to a $T$-invariant function on $\widetilde{\cC}_\sigma$. This is Section \ref{section:step3}, which relies on Demazure resolutions that we study in Appendix \ref{demazure-resolutions}. The core arguments about extending functions are actually in Section \ref{Extending functions}.
\end{enumerate}
In Sections \ref{section:upper-bounds} and \ref{section:step3} as well as Appendix \ref{demazure-resolutions}, we actually keep the assumption that $K = \bQ_p$ to simplify the main arguments.
The extension to the general case is done in Section \ref{section:general-case}.
This is largely just a matter of extra bookkeeping, and reducing to the case $K = \bQ_p$ where possible.

\subsection{Notation}\label{NotationEG}
Throughout the paper, we fix a prime $p$. $K$ denotes a fixed finite unramified extension of $\Q_p$, with ring of integers $\cO_K$. Fix an algebraic closure $\ovl K$ of $K$. We use $G_K$ to denote the Galois group $\Gal(\ovl K/K)$.
The ``coefficient field'' $E$ is a finite extension of $\Q_p$ with residue field $\F$.
We will assume throughout that $E$ is sufficiently large.

Let $G$ denote a split connected reductive group (over some ring) together with a Borel $B$, a maximal split torus $T \subset B$. Let $\Phi^{+} \subset \Phi$ (resp. $\Phi^{\vee, +} \subset \Phi^{\vee}$) denote the subset of positive roots (resp.~positive coroots) in the set of roots (resp.~coroots) for $(G, B, T)$. 
Let $\Delta$ (resp.~$\Delta^{\vee}$) be the set of simple roots (resp.~coroots).
Let $X^*(T)$ be the group of characters of $T$ and $\Lambda_R \subset X^*(T)$ denote the root lattice for $G$.

Set $X^0(T)$ to be the subgroup consisting of characters $\lambda\in X^*(T)$ such that $\langle\lambda,\alpha^\vee\rangle=0$ for all $\alpha\in \Delta$.

Let  $W(G)$ denote the Weyl group of $(G,T)$.  Let $\textrm{w}_0$ denote the longest element of $W(G)$.
We sometimes write $W$ for $W(G)$ when there is no chance for confusion. 
Let $W_a$ (resp.~$\tld{W}$) denote the affine Weyl group and extended affine Weyl group 
\[
W_a = \Lambda_R \rtimes W(G), \quad \tld{W} = X^*(T) \rtimes W(G)
\]
for $G$.
We use $t_{\nu} \in \tld{W}$ to denote the image of $\nu \in X^*(T)$. 

Recall that the affine Weyl group of $\SL_n$ is a Coxeter group generated by $\{s_{\alpha_{i(i+1)}}, s_a|1\leq i\leq n-1\}$, where the former are the spherical reflections and $s_a$ is the affine reflection. We fix lifts of these in $L\SL_n$, the loop group of $\SL_n$ (see also \ref{Affine flag varieties}), as follows:
 \[ s_{\alpha_{i(i+1)}} = \begin{pmatrix}
1 & & & & & \\
& \ddots & & & & \\
& & 0 & - 1 & & \\
& & 1 & 0 & & \\
& & & & \ddots & \\
& & & & & 1
\end{pmatrix}, \,\,\,\,\,
s_a = \begin{pmatrix}
0 & & & & v \\
& 1 & & & \\
& & \ddots & & \\
& & & 1 & \\
- v^{-1} & & & & 0
\end{pmatrix}
\]
and by abuse of notation denote them by the same expression.

Later on we will choose an element $\tld{z}$ with a reduced expression $s_1\cdots s_m$. We fix the lift of $\tld{z}$ which is the product of the lifts of $s_1\cdots s_m$.

Different choices of lifts do not affect double cosets like $I\tld{z}I$, where $I$ is the Iwahori subgroup of $L\GL_n$ as defined in \ref{Affine flag varieties}. But in constructing global functions like $f:I\tld{z}I\to \bA_1$ we need to fix such a lift. 

\subsection{Acknowledgment}
We would like to thank Bao Le Hung for suggesting the problem to us and for providing invaluable guidance during the project.
We would also like to thank Yulia Kotelnikova and David Marcil for their contributions in early stages of the project.
We are grateful to Ana Caraiani and Brandon Levin for valuable comments on the first draft of the paper.
We are also grateful to Dat Pham for finding an error in the proof of Theorem \ref{important} in the first version of the paper available on arXiv, and for suggesting a correction to the same proof.
Also, we would like to thank Fred Diamond and Toby Gee for several comments and suggestions in the first version of the paper.
Finally, we would like to thank the anonymous referees for their careful reading of the manuscript, and for numerous helpful suggestions. 

This work was born out of the NSF-FRG Collaborative Grant DMS-$1952556$. L. J. was supported by the Royal Society Research Grant for Research Fellows 180025. 
\section{Irreducible components of the Emerton-Gee stack}
\label{section:step0}
Let $\mathcal{X}_n = \mathcal{X}_{K,n}$ denote the Emerton-Gee stack of rank $n$ \'etale $(\varphi,\Gamma)$-modules, as defined in \cite[Definition 3.2.1]{emerton-gee}.
Its underlying reduced substack $\mathcal{X}_{n,\text{red}}$ is an algebraic stack of finite type over $\mathbb{F}$, and the irreducible components of $(\mathcal{X}_{n,\text{red}})_{\F}$ are labeled by the Serre weights of $\GL_n(\mathcal{O}_K)$ \cite[Theorem 5.5.12, Theorem 6.5.1]{emerton-gee}.
Given a Serre weight $\sigma$, let $\cX(\sigma) \subseteq (\mathcal{X}_{n,\text{red}})_{\bF}$ be the irreducible component labeled by $\sigma$.
The starting point for our computations in later sections is a description of $\mathcal{X}(\sigma)$ as a quotient of an explicit closed subscheme of an affine flag variety over $\F$, given by \cite[Theorem 7.4.2]{localmodels}.
We will now introduce some notation and recall this description of $\mathcal{X}(\sigma)$.
The reader may treat this as a guide to the relevant sections of \cite{localmodels}.

In what follows, all constructions having to do with Serre weights and related combinatorics happen on the ``representation theoretic side'', whereas the flag varieties we consider are on the ``Galois side''.
The same holds throughout the paper, with the notable exception of Appendix \ref{demazure-resolutions}.
When relating data on one side to data on the other side, we will sometimes use the anti-involution $* : \GL_n \to \GL_n$ which takes a matrix to its transpose, and has a corresponding interpretation on Weyl groups \cite[Definition 2.1.9]{localmodels}.
Morally, the appearance of this anti-involution can be explained by the fact that the flag varieties in \cite{localmodels} model \'etale $\varphi$-modules for the \emph{contravariant} versions of Fontaine functors in $p$-adic Hodge theory.

\subsection{Serre weights}
\label{section:serre-weights} 

This section follows \cite[Section 9]{gee-herzig-savitt} and \cite[Section 2.2]{localmodels}.

A \textit{Serre weight} for $\GL_n(\mathcal{O}_K)$ is an irreducible representation $\sigma: \GL_n(k) \to \GL_{n'}(\F)$. By assumption the field $E$ is sufficiently large, so all irreducible representations of $\GL_n(k)$ over $\ovl{\F_p}$ are defined over $\F$.
Serre weights are parameterized by their highest weight vectors as we will see shortly. 

In order to classify Serre weights we will need some auxiliary groups.
Let $G \defeq \Res_{\F_p}^{k} \GL_n$ with maximal torus $T_G \defeq \Res_{\F_p}^k T$ and Borel $B_G \defeq \Res_{\F_p}^k B$.
In addition, let $\underline{G} \defeq G_{\F} \cong \prod_{k \hookrightarrow \F} \GL_n$, $\underline{T} \defeq (T_G)_{\F}$ and $\underline{B} \defeq (B_G)_{\F}$. Then we can also define the finite Weyl group $\underline{W} \defeq N_{\underline{G}}(\underline{T})/\underline{T}$ . 
We let $\underline{\widetilde{W}} \defeq X^*(\underline{T}) \rtimes \underline{W} \cong \prod_{\iota : \F_p \hookrightarrow k} \widetilde{W}$ denote the extended affine Weyl group for $(\underline{G},\underline{T})$.
Let $\underline{\Phi}, \underline{\Phi}^+, \underline{\Delta}$ denote the roots, positive roots, and simple roots, respectively. 

Let $\mu \in X^*(\underline{T})$ be a dominant character which is $p$-restricted, i.e. it lies in $X_1(\underline{T}) = \lbrace \nu \in X^*(\underline{T}) | 0 \leq \langle \nu, \alpha^\vee \rangle \leq p-1 \text{ for all }\alpha \in \Delta \rbrace$. We let $F(\mu) = \soc \Ind_{\underline{B}}^{\underline{G}} \textrm{w}_0 \mu$ denote the irreducible algebraic representation of $\underline{G}$ with highest weight $\mu$ (see also \cite[Section II.2]{jantzen} and \cite[Corollary 3.17]{herzig09}).
Then $F(\mu)|_{G(\F_p)} = F(\mu)|_{\GL_n(k)}$ is irreducible and thus is a Serre weight. 
In fact, every Serre weight is of this form, but we can give a more precise statement.
The map
\begin{align*}
\frac{X_1(\underline{T})}{(p-\pi)X^0(\underline{T})} & \to \lbrace \text{Serre weights of }\GL_n(k) \rbrace \\
\mu & \mapsto F(\mu)|_{\GL_n(k)}
\end{align*}
is a bijection, where the $\pi(\nu)_\iota = \nu_{\iota \circ \varphi^{-1}}$ with respect to the decomposition $X^*(\underline{T}) \cong \prod_{\iota : \F_p \hookrightarrow k} X^*(T_k)$, $\varphi(x) = x^p$ being the absolute Frobenius \cite[Lemma 9.2.4]{gee-herzig-savitt} \cite[Appendix, Proposition 1.3]{herzig09}.

We say that a Serre weight $\sigma = F(\mu)$, where $\mu \in X_1(\underline{T})$, is \textit{$m$-deep} if $\mu$ is $m$-deep in its $p$-alcove in the sense of \cite[Definition 2.1.10]{localmodels}. Note that this is well-defined because translations by elements of $X^0(\underline{T})$ preserve depth.

\subsubsection{Lowest alcove presentations}
\label{sec:lowest-alcove-presentations}
Let $\underline{\widetilde{W}} \defeq X^*(\underline{T}) \rtimes \underline{W} \cong \prod_{\iota : \F_p \hookrightarrow k} \widetilde{W}$ denotes the extended affine Weyl group for $(\underline{G},\underline{T})$.
Let $\eta \in  X^*(\underline{T})$ denote an element lifting the half sum of positive roots of $\PGL_n$, and let $$ \underline{C}_0 \defeq \lbrace \lambda \in X^*(\underline{T}) \otimes_{\Z} \R | 0 < \langle \lambda + \eta , \alpha^\vee \rangle < p \text{ for all }\alpha \in \underline{\Phi}^+ \rbrace $$
denote the dominant base $p$-alcove. The $p$-dot action of $\underline{\widetilde{W}}$ on $X^*(T) \otimes_{\Z}\R$ is given by
$$ \widetilde{w} \cdot \lambda = w(\lambda + \eta + p \nu) - \eta , \,\,\,\,\, \text{ for }\widetilde{w} = w t_\nu \in \underline{\widetilde{W}} \text{ and } \lambda \in X^*(\underline{T}) \otimes_\Z \R .$$
Finally, let 
\[\underline{\widetilde{W}}_1^+ = \lbrace \widetilde{w} \in \underline{\widetilde{W}} | \widetilde{w} \cdot \underline{C}_0 \text{ is }p\text{-restricted} \rbrace = \lbrace \widetilde{w} \in \underline{\widetilde{W}} | 0 \leq \langle \widetilde{w} \cdot \lambda, \alpha^\vee \rangle < p-1 \text{ for all }(\lambda,\alpha) \in \underline{C}_0 \times \Delta \rbrace .\]

For $\omega - \eta \in \underline{C}_0 \cap X^*(\underline{T})$ and $\widetilde{w} \in \underline{\widetilde{W}}_1^+$, define
$$ F_{(\widetilde{w},\omega)} = F(\pi^{-1} (\widetilde{w}) \cdot (\omega - \eta)), $$
where $\pi: \underline{\widetilde{W}} \to  \underline{\widetilde{W}}$ is the natural extension of the automorphism $\pi$ on $X^*(\underline{T}) $ defined before.
Then $F_{(\bullet,\bullet)}$ factors through the equivalence relation $(\widetilde{w},\omega) \sim (t_\nu \widetilde{w}, \omega - \nu)$, $\nu \in X^0(\underline{T})$.
The equivalence class of $(\widetilde{w},\omega)$ is called a \textit{lowest alcove presentation} of $F_{(\widetilde{w},\omega)}$.
\begin{rmk}
Note that since the $p$-dot action preserves depth, the following are equivalent \cite[Definition 2.1.10]{localmodels}:
\begin{enumerate}
\item $\sigma$ is $m$-deep;
\item $\omega - \eta$ is $m$-deep;
\item $\omega$ is $m$-generic.
\end{enumerate}\end{rmk}
Intuitively, a lowest alcove presentation of $\sigma = F(\mu)$ keeps track of the alcove containing $\mu$ (which by definition is the alcove containing $\sigma$).
Because the affine Weyl group $\underline{W}_a = \langle \underline{\Phi}\rangle \rtimes \underline{W}$ acts simply transitively  on the set of alcoves, there exists a unique\footnote{As long as $\mu$ is 0-deep.} pair $(\tld{w},\omega - \eta) \in \underline{W}_a \times \underline{C}_0$ such that $\pi^{-1}(\tld{w}) \cdot (\omega - \eta) = \mu$.
\begin{convention}
\label{serre-weight-convention}
For $\sigma$ a (0-deep) Serre weight, we always work with the (unique) lowest alcove presentation $(\tld{w},\omega)$ where $\tld{w} \in \underline{W}_a$.
\end{convention}
\begin{rmk}
In the terminology of \cite[Section 2.2]{localmodels}, this convention amounts to saying that we always choose a lowest alcove presentation compatible with the trivial central character $\zeta = 0$. Note that by \cite[Lemma 2.2.4]{localmodels} this does not impose any constraints on the Serre weight.

\end{rmk}
It will be convenient to have the following definition.
\begin{defn} \label{*}
    Given a Serre weight $\sigma$ with lowest alcove presentation $(\tld{w},\omega)$, let $\tld{z} \defeq \textrm{w}_0 \tld{w}$.
    Consequently, 
    $$\tld{z}^* = (\textrm{w}_0 \tld{w})^* = \tld{w}^* \textrm{w}_0 = w^{-1} t_\nu \textrm{w}_0 = t_{w^{-1} \nu} w^{-1} \textrm{w}_0 , $$
    where $*$ denotes the anti-involution of \cite[Definition 2.1.9]{localmodels} that was mentioned in the beginning of this section.
\end{defn}
\subsection{Affine flag varieties}\label{Affine flag varieties}
Our affine flag varieties are defined in terms of the following functors on $\mathbb{F}$-algebras.
Some of them will not appear until Appendix \ref{demazure-resolutions}, but for convenience we list them all here. Let $R$ be an $\mathbb F$-algebra.
\begin{align*}
L \GL_n (R) & \defeq \GL_n(R((v))) & \text{( loop group)} \\
 L^+ \GL_n (R) & \defeq \cK= \GL_n(R[[v]]) & \text{( positive loop group)} \\
I(R) & \defeq \lbrace A \in L^+ \GL_n(R) | A \text{ is upper triangular mod } v \rbrace & \text{( Iwahori)} \\
\bar I(R) & \defeq \lbrace A \in L^+ \GL_n(R) | A \text{ is lower triangular mod } v  \rbrace & \text{( opposite Iwahori)} \\
I_1(R) & \defeq \lbrace A \in L^+ \GL_n(R) | A \text{ is unipotent upper triangular mod } v \rbrace & \text{(the pro-}v\text{ Iwahori)} \\
\bar I_1(R) & \defeq \lbrace A \in L^+ \GL_n(R) | A \text{ is unipotent lower triangular mod }v \rbrace & \text{(the pro-}v\text{ opposite Iwahori)} \\
T(R) & \defeq \lbrace \text{diagonal matrices in } \GL_n(R) \rbrace & \text{( constant torus)} \\
\Lie I(R) & \defeq \lbrace A \in \Mat_{n \times n}(R[[v]]) | A \text{ is upper triangular mod }v \rbrace & \text{( Lie algebra of }I\text{)} .
\end{align*}
The affine flag variety $\Fl$ is defined as the fpqc quotient sheaf $\Fl \defeq I \backslash L \GL_n$. Similarly, we define $\widetilde{\Fl} \defeq I_1 \backslash L \GL_n$.
Since $I = I_1 \rtimes T$ it follows that $\widetilde{\Fl} \to {\Fl}$ is a $T$-torsor, where $T$ acts by left translation.

Let $\mathcal{J} \defeq \Hom_{\F_p}(k,\F)$, and define $\Fl_\mathcal{J} \defeq \prod_{\iota \in \mathcal{J}} \Fl$ and $\widetilde{\Fl}_{\mathcal{J}} \defeq \prod_{\iota \in \mathcal{J}} \widetilde{\Fl}$.
Now $\widetilde{\Fl}_{\mathcal{J}} \to \Fl_{\mathcal{J}}$ is a left $T^{\mathcal{J}}$-torsor, where $T^{\mathcal{J}}$ acts by left translation in each factor.
\begin{rmk}
    Affine flag varieties are not varieties. In fact, they are ind-schemes.  They are the affine version of flag varieties, as we replace $G$ by its loop group $LG$ and the Borel subgroup by the Iwahori subgroup. 
\end{rmk}

\subsubsection{Monodromy condition}
Let $\nabla_0$ denote the following monodromy condition $A \in L\GL_n(R)$ \cite[Section 4.3]{localmodels}:
$$ v \left( \frac{d}{dv} A \right) A^{-1} \in \frac{1}{v} \Lie I(R) . $$
Let $(L \GL_n)^{\nabla_0} (R) = \lbrace A \in L \GL_n (R) | A \text{ satisfies the }\nabla_0\text{-condition.} \rbrace$.
This condition is preserved under left translation by $I(R)$, so the quotient $I(R) \backslash L \GL_n^{\nabla_0}(R)$ is well defined.
Let $\Fl^{\nabla_0}$ denote the fpqc-sheafification of the functor $R \mapsto I(R) \backslash L \GL_n^{\nabla_0}(R)$.
We also define $\widetilde{\Fl}^{\nabla_0}$ as the preimage of $\Fl^{\nabla_0}$ in $\widetilde{\Fl}$. This is equivalently the fpqc-sheafification of the functor $R \mapsto I_1(R) \backslash L \GL_n^{\nabla_0}(R)$.

\subsubsection{Schubert varieties}
Given $\widetilde{w},\widetilde{s} \in \underline{\widetilde{W}} = \prod_{\iota \in \cJ} \widetilde{W}$, we make the following definitions: 
\begin{enumerate}
\item $S^\circ (\widetilde{w}) = \prod_{\iota \in \cJ} I \backslash I \tld{w}_\iota I \subset \Fl_\cJ$ is the \textit{open Schubert cell} associated to $\widetilde{w}$. This is the fpqc sheafification of the functor $R \mapsto \prod_{\iota \in \cJ} I(R) \backslash I(R) \widetilde{w}_\iota I(R)$ on $\F$-algebras $R$. 
Similarly, let $\tld{S}^\circ(\tld{w}) = \prod_{\iota \in \cJ} I_1 \backslash I_1 \tld{w}_\iota I \subset \tld{\Fl}_\cJ$.
\item Let $S^\circ(\tld{w},\tld{s}) = \prod_{\iota \in \cJ} I \backslash I \tld{w}_\iota I \tld{s}_\iota \subset \Fl_\cJ$ and $\tld{S}^\circ(\tld{w},\tld{s}) = \prod_{\iota \in \cJ} I_1 \backslash I_1 \tld{w}_\iota I \tld{s}_\iota \subset \tld{\Fl}_\cJ$ be translates of the open Schubert cells from before.
\item The Schubert variety $S(\widetilde{w})$ is the closure of $S^\circ(\widetilde{w})$ in $\Fl_\cJ$, and $\tld{S}(\tld{w})$ is the closure of $\tld{S}^\circ(\tld{w})$ in $\tld{\Fl}_\cJ$. Similarly, $S(\tld{w},\tld{s})$ is the closure of $S^\circ(\tld{w},\tld{s})$ in $\Fl_\cJ$, and $\tld{S}(\tld{w},\tld{s})$ is the closure of $\tld{S}^\circ(\tld{w},\tld{s})$ in $\tld{\Fl}_\cJ$.
\item Whenever we put a superscript $\nabla_0$ it means the locus satisfying the monodromy condition. For example, $S(\tld{w},\tld{s})^{\nabla_0} = S(\tld{w},\tld{s}) \cap \Fl_{\cJ}^{\nabla_0}$.
\item Define $S^{\nabla_0}(\tld{w},\tld{s})$ as the closure of $S^\circ(\tld{w},\tld{s})^{\nabla_0}$ in $\Fl^{\nabla_0}_\cJ$. Similarly, define $\tld{S}^{\nabla_0}(\tld{w},\tld{s})$ as the closure of $\tld{S}^\circ(\tld{w},\tld{s})^{\nabla_0}$ in $\tld{\Fl}^{\nabla_0}_\cJ$.
\end{enumerate}
\begin{rmk}
 The notation introduced here differs from that of \cite[Definition 4.3.2]{localmodels}. For example, what would be called $S^{\circ}(\tld{w}_1,\tld{w}_2,\tld{s})$ in loc. cit. corresponds to $S^{\circ}((\tld{w}_2^{-1} \textrm{w}_0 \tld{w}_1)^*,\tld{s}^*)$ in our notation.
\end{rmk}

\begin{exam}\label{schubert-cell-discrepancy}
    $S^{\nabla_0}(\widetilde{w},\widetilde{s})$ is typically smaller than $S(\tld{w},\tld{s})^{\nabla_0}$ \cite[Remark 4.3.3(2)]{localmodels}.
    Let us give an example.
   Let $G=\GL_{2,\Q_p}$ and consider the element $\widetilde{w}=\begin{pmatrix}
       0 & -v^{-1} \\
      v & 0
   \end{pmatrix} =s_{12}\cdot s_a \cdot s_{12}$ (cf. \ref{NotationEG}). Then one can compute that $S^\circ (\widetilde{w})\cong \bA^3$ and the monodromy condition cuts out a
   two-dimensional subspace $S^\circ(\tld{w})^{\nabla_0}\cong \bA^2\subset \bA^3$. The closure of the Schubert variety is given by $S(\widetilde{w})= \coprod_{ \tld{w}' \leq \tld{w}} S^\circ (\tld{w})$. 
   More precisely the elements $\tld{w}'\leq \tld{w}$, where we use the Bruhat order, are as follows
\[\begin{tikzcd}
   & {s_{12}\cdot s_a} & {s_{12}} \\
   {s_{12}\cdot s_a\cdot s_{12}} &&& 1 \\
   & {s_a\cdot s_{12}} & {s_a}
   \arrow[from=1-2, to=2-1]
   \arrow[from=3-2, to=2-1]
   \arrow[from=1-3, to=1-2]
   \arrow[from=1-3, to=3-2]
   \arrow[from=3-3, to=3-2]
   \arrow[from=3-3, to=1-2]
   \arrow[from=2-4, to=1-3]
   \arrow[from=2-4, to=3-3]
\end{tikzcd}.\]
Now one can compute that for $\tld{w}'<\tld{w}$ we have $S^\circ (\tld{w}')^{\nabla_0}=S^\circ (\tld{w}')\cong \bA^{l(\tld{w}')}$, i.e. the monodromy condition doesn't affect the open Schubert cell.
This is because, up to left multiplication by $I$, the entries in the matrices have degree $1$ or lower (we refer to Example \ref{toy_case} to see how one would exactly compute the monodromy condition). To summarize, we have
\begin{enumerate}
   \item $S(\tld{w},\tld{s})^{\nabla_0}= \coprod_{ \tld{w}' \leq \tld{w}} S^\circ (\tld{w})^{\nabla_0}\cong \bA^{2} \coprod (\bA^{2} \coprod \bA^{2}) \coprod (\bA^{1} \coprod \bA^{1}) \coprod \bA^{0} $.
   \item $S^{\nabla_0}(\widetilde{w},\widetilde{s})$ is the closure of $S^\circ(\tld{w})^{\nabla_0}\cong \bA^{2}$ inside $S(\tld{w},\tld{s})^{\nabla_0}$. Now $$\bA^{2}\subset \bA^{2} \coprod (\bA^{2} \coprod \bA^{2}) \coprod (\bA^{1} \coprod \bA^{1}) \coprod \bA^{0}$$ cannot be dense because of the dimensions of the substrata (see eg. \cite[Tag 0BCQ]{stacks-project}), so we see that the closure is strictly smaller.
\end{enumerate}
\end{exam}

\subsubsection{Shifted conjugation}
In addition to left translation, there is another action of $T^{\cJ}$ on $\widetilde{\Fl}^{\cJ}$ that will be important to us, namely the \textit{shifted conjugation action} \cite[Section 5.4]{localmodels}. This is induced by the following action of $T^\cJ$ on $L\GL_n^\cJ = \prod_{\iota \in \cJ} L \GL_n$: Given $D = (D_\iota)_{\iota \in \cJ} \in T^\cJ(R)$ and $A = (A_\iota)_{\iota \in \cJ} \in L\GL_n^{\cJ}(R)$, the shifted conjugation action is given by
$$ (D \cdot A)_\iota = D_\iota A_\iota D_{\iota \circ \varphi}^{-1} .$$
Note that if we choose an embedding $\iota_0 : k \hookrightarrow \F$ and order the embeddings as $(\iota_0,\iota_1,\dots, \iota_{r-1}) = (\iota_0,\iota_0 \circ \varphi^{-1},\dots,\iota_0 \circ \varphi^{-f+1})$, then we can write the shifted conjugation action as
$$ (D_0,D_1,\dots,D_{f-1}) \cdot (A_0,A_1,\dots,A_{f-1}) = (D_0 A_0 D_{f-1}^{-1},D_1 A_1 D_0^{-1},\dots,D_{f-1} A_{f-1} D_{f-2}^{-1}) ,$$
where we use the shorthand $D_i = D_{\iota_i}, A_i = A_{\iota_i}$.

Given $D = (D_\iota)_{\iota \in \cJ} \in T^\cJ(R)$ and $A = (A_\iota)_{\iota \in \cJ} \in L\GL_n^{\cJ}(R)$, we can see that $$v\frac{d(D_\iota A_\iota D^{-1}_{\iota \circ \phi})}{dv}(D_\iota A_\iota D_{\iota \circ \phi}^{-1})^{-1}=vD_\iota \frac{ dA_\iota }{dv}D^{-1}_{\iota \circ \phi}(D_\iota A_\iota D_{\iota \circ \phi}^{-1})^{-1}=D_\iota(v  \frac{dA_\iota}{dv} A_\iota^{-1} )D_\iota^{-1},$$ hence the shifted conjugation preserves the subvariety cut out by the monodromy condition.

Note that the open Schubert cells $\tld{S}^\circ(\tld{w},\tld{s})$ and the Schubert varieties $\tld{S}(\tld{w},\tld{s})$ are stable under the shifted conjugation action.
Indeed, suppose $\tld{w} = t_\nu w$ and $\tld{s} = s t_\mu$, where $t_\nu,t_\mu$ denote translations by $\nu,\mu$ respectively, and $w,s \in \underline{W}$.
Given $D =(D_\iota)_{\iota \in \cJ} \in T^\cJ(R)$ and $A = (A_\iota)_{\iota \in \cJ} \in \prod_{j \in \cJ}I_1 \widetilde{w}_\iota I \tld{s}_\iota(R)$, we can choose the coset representative $A_\iota$ to have the form $\widetilde{w}_\iota \alpha_\iota \tld{s}_\iota$ with $\alpha_\iota\in I(R)$. Then $D_\iota A_\iota D_{\iota\circ \phi}^{-1}=D_\iota \widetilde{w}_\iota  \alpha_\iota \tld{s}_\iota D_{\iota\circ \phi}^{-1}=\widetilde{w}_\iota \Ad_{\tld{w}_\iota^{-1}}(D_\iota) \alpha_\iota \Ad_{\tld{s}_\iota}(D_{\iota\cdot \phi}^{-1}) \tld{s}_\iota \in I_1 \widetilde{w}_\iota I \tld{s}(R)$ since $\Ad_{\tld{w}_\iota^{-1}}(D_\iota) = \Ad_{w_\iota^{-1}}(D_\iota), \Ad_{\tld{s}}(D_{\iota \circ \varphi}^{-1}) = \Ad_{s_\iota}(D_{\iota \circ \varphi}^{-1}) \in T(R) \subseteq I(R)$.

\subsection{Presentation of the irreducible components}\label{presirred}
\label{section:presentation-of-the-irreducible-components}
Let $\sigma$ be a Serre weight for $\GL_n(\cO_K)$ which is at least $(n-1)$-deep and with lowest alcove presentation $(\tld{w},\omega) = (t_\nu w, \omega)$.
Recall from Convention \ref{serre-weight-convention} that $\tld{w}$ is an element of the affine Weyl group $\underline{W}_a$.
In addition, recall from Definition \ref{*} that $\tld{z} = \textrm{w}_0 \tld{w}$ and therefore $\tld{z}^* = \tld{w}^* \textrm{w}_0 = w^{-1} t_\nu \textrm{w}_0$.

Following \cite[Sections 4.4-4.6]{localmodels} we make the following definition.
\begin{defn}
Let $\sigma$ be as above. We define the following the subschemes
\begin{align*}
\mathcal{C}_\sigma & \defeq S^{\nabla_0}(\tld{z}^*,t_\omega) \subset \Fl^{\nabla_0}_\cJ \\ 
\widetilde{\cC}_\sigma & \defeq \tld{S}^{\nabla_0}(\tld{z}^*,t_\omega) = \text{the preimage of }\cC_\sigma\text{ in }\widetilde{\Fl}^{\nabla_0}_\cJ .
\end{align*}
\end{defn}
\begin{rmk}
\label{a-irreducible}
\begin{itemize}
    \item This does not depend on the choice of representative $(\widetilde{w},\omega)$. If we take another representative $(t_{\lambda}\widetilde{w},\omega-\lambda)$ where $\lambda\in X^0(\underline{T})$, then the equality
     $$ I \backslash I \tld{z}_\iota^* I v^{\omega_\iota} = I \backslash I \tld{z}_\iota^* I v^{\lambda_{\iota}} v^{-\lambda_{\iota}} v^{\omega_\iota}
     = I \backslash I w^{-1} t_\lambda v^{\lambda_{\iota}}\textrm{w}_0 I  v^{\omega_\iota-\lambda_{\iota}}
     $$  
     yields the desired result. Here $v^{\omega}=\prod_{\iota\in \cJ} v^{\omega_{\iota}}$ is cocharacter which is dual to the character $\omega$. 
     In particular we used that the dual of $\lambda$ is central $v^{\lambda}$, since $\lambda\in X^0(\underline{T})$ (see also \cite[Section 3.3]{herzig09}).
    \item Both $\mathcal{C}_\sigma$ and $\widetilde{\cC}_\sigma$ are irreducible closed subschemes of dimension $(\# \cJ)d$ and $(\# \cJ)(n + d)$, respectively, where $d = \dim_\F(B \backslash \GL_n)_\F$ \cite[Section 4.5]{localmodels}.
\end{itemize}
\end{rmk}
Let us spell out in more detail what these objects are.
By definition, $\cC_\sigma = S^{\nabla_0}(\widetilde{z}^*,t_{\omega})$ is the closure of 
$$ U_1 \defeq S^\circ(\widetilde{z}^*, t_{\omega})^{\nabla_0} = \left( \prod_{\iota \in \cJ} I \backslash I \tld{z}_\iota^* I v^{\omega_\iota} \right) \cap \Fl^{\nabla_0}_{\cJ} $$
inside $\Fl^{\nabla_0}_\cJ$.
Similarly, $\widetilde{\cC}_\sigma = \tld{S}^{\nabla_0}(\tld{z}^*,t_\omega)$ is the closure of 
$$ \tld{U}_1 \defeq \tld{S}^\circ(\tld{z}^*,t_\omega)^{\nabla_0} = \left( \prod_{\iota \in \cJ} I_1 \backslash I_1 \tld{z}_\iota^* I v^{\omega_\iota} \right) \cap \tld{\Fl}^{\nabla_0}_{\cJ} $$
inside $\tld{\Fl}^{\nabla_0}_\cJ$.
\begin{thm}
\label{presentation-of-EG-component}
Assume that $\sigma$ is a $(3n-1)$-deep Serre weight. Then $$\cX(\sigma) \cong \left[ \widetilde{\mathcal{C}}_\sigma / T^\cJ \right],$$
where $T^\cJ$ acts by shifted conjugation.
\end{thm}
\begin{proof}
This is \cite[Remark 7.4.3(2)]{localmodels}. 
\end{proof} 

\begin{rmk}
    The label $\sigma$ of the component $\cX(\sigma)$ is chosen as to be compatible with \cite[Page 115]{localmodels}.
    This is the component labeled by $\sigma^\vee \otimes \det^{n-1}$ according to the conventions of \cite{emerton-gee}.
\end{rmk}

\begin{cor}
\label{global-functions-as-invariant-sections}
With the same assumptions as above, the ring of global functions on $\cX(\sigma)$ is given by the $T^\cJ$-invariant sections $\cO(\tld{\cC}_\sigma)^{T^\cJ}$, where $T^\cJ$ acts by shifted conjugation.
\end{cor}

\section{Open charts}
\label{section:open-charts}
In this section, we will introduce certain subschemes of $\tld{\cC}_\sigma$ that will be used to compute the $T^\cJ$-invariant sections.

As in Section \ref{section:presentation-of-the-irreducible-components}, $\sigma$ denotes an $(n-1)$-deep Serre weight with lowest alcove presentation $(\tld{w},\omega)$ and $\tld{z} \defeq \textrm{w}_0 \tld{w}$.
Let $\cK \defeq L^+\GL_n$.
\begin{defn} \label{Uchart}
For $s \in \underline{W}$, define the following subschemes
$$ \tld{U}_s \defeq \tld{S}^\circ(\tld{z}^*,s t_\omega)^{\nabla_0} = \prod_{\iota \in \cJ} I_1 \backslash I_1 \tld{z}_\iota^* I s_\iota v^{\omega_\iota} \cap \tld{\Fl}^{\nabla_0}_\cJ $$
and 
$$  \tld{U} \defeq \prod_{\iota \in \cJ} I_1 \backslash I_1 \tld{z}_\iota^* \cK v^{\omega_\iota} \cap \tld{\Fl}^{\nabla_0}_\cJ. $$
\end{defn}

For our purposes, the key properties of $\tld{U}_s$ are:
\begin{enumerate}[(A)]
\item The various $\tld{U}_s$ are open and dense subschemes of $\tld{\cC}_\sigma$;
\item The collection $\{\widetilde{U}_{s}\}_{s\in \underline{W}}$ is a Zariski open cover of $\widetilde{U}$;
\item For each $s \in \underline{W}$ the open $\tld{U}_s$ can be given natural coordinates in upper triangular matrices.
\end{enumerate}
Before proving these properties, let us first establish a consequence of them.
\begin{lemma}
\label{large-open}
The scheme $\tld{U}$ in $\tld{\Fl}^{\nabla_0}_\cJ$ is irreducible.
\end{lemma}
\begin{proof}
By property (A) we know that the $\tld{U}_s$ are open inside $\widetilde{\cC}_\sigma$. By property (B), the $\tld{U}_s$ cover $\tld{U}$. Therefore $\tld{U}$ is open inside $\widetilde{\cC}_\sigma$ and irreducible itself (see also Remark \ref{a-irreducible}).
\end{proof}
\begin{rmk}
Arguing as in \cite[Section 2]{alvarez}, it should be possible to prove Lemma \ref{large-open} directly.
\end{rmk}

\begin{lemma}[Property A]
\label{property-A}
For each $s \in \underline{W}$, we have $\tld{U}_s \subset \tld{\cC}_\sigma$, and the inclusion
$$\tld{U}_s \hookrightarrow \tld{\cC}_\sigma $$
is open.
\end{lemma}
\textit{Proof.}
By \cite[Proposition 4.3.6]{localmodels}, the closure of $\tld{U}_s$ in $\tld{\Fl}^{\nabla_0}_\cJ$ equals the closure of $\tld{U}_1$ in $\tld{\Fl}^{\nabla_0}_\cJ$, which by definition is $\tld{\cC}_\sigma$.
To see that $\tld{U}_s = \tld{S}^\circ (\tld{z}^*,st_\omega)^{\nabla_0}$ is open in its closure inside $\tld{\Fl}_\cJ^{\nabla_0}$, note first of all that $\tld{S}^\circ(\tld{z}^*,st_\omega)$ is open in $\tld{S}(\tld{z}^*,st_\omega)$ because open Schubert cells are open in their closures.
Hence $\tld{U}_s$ is open in $\tld{S}(\tld{z}^*,st_\omega)^{\nabla_0}$, which is closed in $\tld{\Fl}_\cJ^{\nabla_0}$ and therefore contains the closure of $\tld{U}_s$.
$\Box$

\begin{lemma}[Property B]\label{Property B}
The collection $\{\widetilde{U}_{s}\}_{s\in W}$ is an open cover of $\widetilde{U}$. 
\end{lemma}
\begin{proof}
Consider the map $\pi_v : I \backslash \cK \to B \backslash \GL_n$ given by evaluating at $v=0$, where $B \leq \GL_n$ is subgroup of upper triangular matrices.
Note that the image of $I \textrm{w}_0 I$ under $\pi_v$ is the open Schubert cell $B \textrm{w}_0 B$.
Since the flag variety $B \backslash \GL_n$ is covered by Weyl group translates of the open Schubert cell, we have 
\begin{equation}\label{eq:IK=cup_IwIs}
        I\textrm{w}_0\cK = \bigcup_{s\in W} I\textrm{w}_0Is.
\end{equation} 
By the Demazure factorization \cite[Proposition 2.8]{iwahori1965some} combined with \eqref{eq:IK=cup_IwIs}, 
\begin{equation}\label{eq:IwK=cup_IwIs}
    I\tld{z}^*\cK = I\tld{w}^*I\textrm{w}_0\cK = \bigcup_{s\in W} I\tld{w}^*I\textrm{w}_0Is = \bigcup_{s\in W} I\tld{z}^*Is.
\end{equation}
(Here, we use that $\tld{z}^* = \tld{w}^* \textrm{w}_0$ is a reduced factorization. This is true because $\tld{w} \in \underline{W}_1^+$, implying that $\textrm{w}_0 \tld{w}$ is a reduced factorization, which is preserved by $*$.)
The rest follows from $I_{1}\widetilde{z}^*I$ being a trivial $T$-torsor over $I\widetilde{z}^*I$.
\end{proof}

\subsection{Upper triangular coordinates} 
\label{1+2}
Finally, we turn to ``Property C'', namely the existence of certain coordinates on the charts $\tld{U}_s$.
Similar computations to the ones we do in this section can be found in \cite[Section 4.2]{localmodels}.

We will have to introduce some notation.
\begin{defn}
\label{defined-by-valuation-bounds}
We will say that a subfunctor $X \subset L^+\GL_n$ is \textit{defined by valuation bounds} if $X$ is of the form
$$ X(R) = \left\lbrace A \in L^+ \GL_n(R) : \begin{array}{c} \text{ the }(k,l)\text{'th entry of }A\text{ is of the form } \\ a_{m_{kl}}v^{m_{kl}} + a_{m_{kl}+1}v^{m_{kl}+1} + \cdots + a_{n_{kl}}v^{n_{kl}}\text{, for }a_i \in R \end{array} \right\rbrace ,$$
and in this case we say that $X$ has valuation bounds $[m_{kl},n_{kl}]$ in position $(k,l)$.
\end{defn}
For example, 
$$ X(R) = \left\lbrace
\begin{pmatrix} 
a_{-3}v^{-3} + a_{-2}v^{-2} + a_{-1} v^{-1} & b_{-1} v^{-1} + b_0 \\
0 & d_0
\end{pmatrix}
: a_i, b_i, d_i \in R \right\rbrace \subset L \GL_2(R) $$
is given by valuation bounds $\begin{pmatrix}
[-3,-1] & [-1,0] \\ \emptyset & [0,0]
\end{pmatrix}.$
More generally, a subfunctor $X = \prod_{\iota \in \cJ} X_\iota \subset L \GL_n^\cJ$ is defined by valuation bounds if each $X_\iota$ is defined by valuation bounds.

The positive loop group $L^+ \GL_n$ is defined by valuation bounds $[0,\infty]$ in every entry, and the Iwahori subgroup $I$ is also given by valuation bounds.
In fact, Bruhat-Tits theory gives a convenient way to describe $\Ad_{\tld{t}} (I)$ in terms of valuation bounds for any $\tld{t} \in \tld{\underline{W}}$, as follows.
The group $\GL_n(k((v)))$ acts on the enlarged building, which contains the standard apartment $X_*(T) \otimes_\Z \R$, and $I$ is the stabilizer of the antidominant base alcove $\textrm{w}_0 A_0$, where  $$ A_0 = \lbrace x \in X^*(T) \otimes_\Z \R : 0 < \langle \alpha, x \rangle < 1 \text{ for all } \alpha \in \Phi^+ \rbrace$$ is the base alcove.\footnote{Our convention that $I$ is the stabilizer of the antidominant base alcove differs from some of the literature related to Bruhat-Tits theory, for example \cite{kaletha-prasad}, in which $I$ would be the stabilizer of the dominant base alcove. We use this convention because it is better suited for working with the flag variety $\Fl = I \backslash L\GL_n$; in particular, the Bruhat ordering determined by reflections across walls of the antidominant base alcove is compatible with the Schubert stratification on $\Fl$ \cite[Section 3]{faltings-loop-groups}.}
Correspondingly, $\Ad_{\tld{t}}(I)$ is the stabilizer of the alcove $\tld{t} \textrm{w}_0 A_0$.
This stabilizer, in turn, is defined by valuation bounds.
It will be more convenient to phrase these in terms of the lower triangular Iwahori subgroup, which stabilizes the base alcove $A_0$ and is transpose to $I$.
\begin{prop}
\label{iwahori-valuation-bounds}
Let $\ovl{I} \subset L^+ \GL_n$ be the Iwahori subgroup which is lower triangular modulo $v$, and let $\tld{t} \in \tld{W}$.
Suppose $\tld{t} A_0$ is described by inequalities $n_\alpha < \langle \alpha, x \rangle < n_\alpha+1$, where $n_\alpha \in \Z$, for each $\alpha \in \Phi^+$.
Then $\Ad_{\tld{t}}(\ovl{I})$ is defined by valuation bounds
\begin{itemize}
\item $[0,\infty]$ on the diagonal entries.
\item $[n_\alpha+1,\infty]$ in position corresponding to $\alpha$, for $\alpha \in \Phi^+$.
\item $[- n_\alpha,\infty]$ in position corresponding to $-\alpha$, for $\alpha \in \Phi^+$.
\end{itemize}
\end{prop}
\begin{proof}
This is a basic fact in Bruhat-Tits theory. See for example \cite[Axiom 4.1.8]{kaletha-prasad}, keeping in mind the discrepancy between their and our conventions as pointed out in the footnote.
\end{proof}
\begin{rmk}
    By the position corresponding to $\alpha$ we mean the following. Let $\alpha\in \Phi^+$. This comes with an isomorphism $\bG_a\xrightarrow{\sim}U_{\alpha}$.
    The subgroup $U_{\alpha}\subset L\GL_n$ will be non-zero at exactly one matrix entry with coordinates $(i,j)$. These coordinates are the position of $\alpha$. 
\end{rmk}

\begin{lemma}
\label{conjugated-iwahori-description}
The Iwahori subgroup $\Ad_{\tld{z}^{-*}}(I^\cJ)$, where $\tld{z}^{-*} = (\tld{z}^*)^{-1}$ is shorthand, is defined by the following valuation bounds:
\begin{itemize}
    \item On the diagonal entries: $[0,\infty]$
    \item On the superdiagonal entries (the entries 1 above the diagonal): $[1,\infty]$
    \item On the remaining upper triangular entries: $[m_{\alpha,\iota}+1,\infty]$ for some $m_{\alpha,\iota} \geq 0$ depending on $\tld{w}_\iota$.
    \item On the subdiagonal entries (the entries 1 below the diagonal): $[0,\infty]$
    \item On the remaining lower triangular entries: $[-m_{\alpha,\iota},\infty]$ for the same $m_{\alpha,\iota} \geq 0$ depending on $\tld{w}_\iota$.
    \end{itemize}
\end{lemma}
\textit{Proof.}
Let us argue in the case $K = \Q_p$ and $\#\cJ = 1$; the general case is then just a matter of changing the notation.

We can describe the transpose as $(\Ad_{\tld{z}^{-*}} I)^* = \Ad_{\tld{z}} \ovl{I}$.
By assumption, $\tld{w} \in \tld{\underline{W}}_1^+$, so $\tld{z} \cdot \underline{C}_0 = \textrm{w}_0 \tld{w} \cdot \underline{C}_0$ is ``anti-$p$-restriced''.
What this means is that the corresponding alcove $\textrm{w}_0 \tld{w}\underline{A}_0$ is described by inequalities of the form $-1 < \langle \alpha, x \rangle < 0$ for $\alpha \in \Delta$.
The result is then clear from Proposition \ref{iwahori-valuation-bounds}. $\Box$

\begin{lemma}
\label{upper-triangular-coordinates}
There exists a subfunctor $\tld{\cA} \subset L^+ B^\cJ$ such that:
\begin{enumerate}
\item The canonical map \begin{align*}
\tld{\cA} & \lrisom \tld{S}^\circ(\tld{z}^*) = \prod_{\iota \in \cJ} I_1 \backslash I_1 \tld{z}_\iota^* I \\
(A_\iota)_{\iota \in \cJ} & \mapsto \left( I_1 \tld{z}_\iota^* A_\iota \right)_{\iota \in \cJ}
\end{align*}
is an isomorphism.
\item $\tld{\cA}$ is defined by valuation bounds, which are:
\begin{itemize}
\item $\emptyset$ in the strictly lower triangular entries. (This is redundant as $\tld{\cA} \subset L^+ B^\cJ$.)
\item $[0,0]$ on diagonal and superdiagonal entries.
\item $[0,m]$ for $m\geq 0$ depending on $\tld{w}$ in all other upper triangular entries.
\end{itemize}
\end{enumerate}
\end{lemma}
\textit{Proof.}
Let us argue in the case $K = \Q_p$ and $\#\cJ = 1$; the general case is then just a matter of changing the notation.
As $I$ acts transitively by right translations on $I \backslash I \tld{z}^* I$, we have by the orbit-stabilizer theorem an isomorphism
\begin{align*}
(I \cap \tld{z}^{-*} I \tld{z}^*) \backslash I & \lrisom I \backslash I \tld{z}^* I = S^\circ(\tld{z}^*) \\
(I \cap \tld{z}^{-*} I \tld{z}^*) A & \mapsto I \tld{z}^* A .
\end{align*}
By Lemma \ref{conjugated-iwahori-description}, $I \cap \tld{z}^{-*} I \tld{z}^*$ is defined by the following valuation bounds:
\begin{itemize}
\item On the diagonal entries: $[0,\infty]$
\item On the superdiagonal entries: $[1,\infty]$
\item On the remaining upper triangular entries: $[m_\alpha + 1,\infty]$ for some $m_\alpha \geq 0$ depending on $\tld{w}$
\item On the strictly lower diagonal entries: $[1,\infty]$
\end{itemize}
By Iwahori decomposition (row operations), the quotient $(I \cap \tld{z}^{-*} I \tld{z}^*) \backslash I$ can be given coordinates $\cA \subset L^+ B^\cJ$ (in the sense that every coset has a unique representative in $\cA$) which is defined by the following valuation bounds:
\begin{itemize}
\item On lower triangular and diagonal entries: $\emptyset$. In this case, that means that the diagonal entries are all equal to $1$.
\item On the superdiagonal: $[0,0]$.
\item On other upper diagonal entries: $[0,m_\alpha]$ for the same $m_\alpha \geq 0$ depending on $\tld{w}$.
\end{itemize}
We can then define $\tld{\cA} \defeq T^\cJ \cA$, and using the fact that $\tld{S}^\circ (\tld{z}^*)$ is a trivial $T^\cJ$-torsor over $S^\circ(\tld{z}^*)$, both (1) and (2) follow. $\Box$

For the remainder of this section, let $\tld{\cA}$ be as in the above lemma.
The next step is to describe $\tld{U}_s$, the locus of $\tld{S}^\circ(\tld{z}^*,st_\omega)$ cut out by $\nabla_0$, in terms of a monodromy condition which is defined directly on $\tld{\cA}$. Towards this end, define for any $s \in \underline{W}$ the monodromy condition $\nabla_{\tld{z}^*,s\omega}$ on $L \GL_n^\cJ$ by
\begin{equation}
\label{advanced-monodromy-condition} 
\nabla_{\tld{z}^*,s\omega} : \,\,\,\,\, v \left( \frac{d}{dv}A \right) A^{-1} + A \operatorname{diag}(s\omega) A^{-1} \in \frac{1}{v} \Lie \Ad_{\tld{z}^{-*}}(I^\cJ) .
\end{equation}
It is straightforward to check that $A \in L \GL_n^\cJ$ satisfies $\nabla_{\tld{z}^*,s\omega}$ if and only if $\tld{z}^* A s v^\omega$ satisfies $\nabla_0$. The following is immediate:
\begin{lemma}
\label{upper-triangular-coordinates2}
Let $s \in \underline{W}$, and let $\tld{A}_s \defeq \tld{A}^{\nabla_{\tld{z}^*,s\omega}}$ be the locus of $\tld{A}$ cut out by $\nabla_{\tld{z}^*,s\omega}$. Then the canonical map
\begin{align*}
\tld{\cA}_s & \lrisom \tld{U}_s = \left( \prod_{\iota \in \cJ} I_1 \backslash I_1 \tld{z}_\iota^* I s_\iota v^{\omega_\iota} \right)^{\nabla_0} \\
(A_\iota)_{\iota \in \cJ} & \mapsto \left( I_1 \tld{z}_\iota^* A_\iota s_\iota v^{\omega_\iota} \right)_{\iota \in \cJ}
\end{align*}
is an isomorphism. Moreover, this map is equivariant for the action of $T^\cJ$ on $\tld{\cA}_s$ given by the formula 
\begin{equation}\label{eq:shifted-conjugation-on-upper-triangular-coordinates} (D \cdot A)_\iota = \Ad_{\mathrm{w}_0 w_\iota}(D_\iota) A_\iota \Ad_{s_\iota}(D_{\iota \circ \varphi})^{-1} \,\,\,\,\, \text{ for } (D,A) \in T^\cJ \times \tld{\cA}_s , 
\end{equation} 
and the shifted conjugation action on $\tld{U}_s$.
\end{lemma}

Observe that the action of $T^\cJ$ given on $\tld{\cA}$ in \eqref{eq:shifted-conjugation-on-upper-triangular-coordinates} takes a particularly nice form for a specific choice of $s$.
\begin{defn}
\label{easy-s}
Let $s_w \in \underline{W}$ be the element $s_w = \textrm{w}_0 \pi^{-1}(w)$. In other words, $s_w$ is the element with $(s_w)_\iota = \textrm{w}_0 w_{\iota \circ \varphi}$.
\end{defn}
Then the action \eqref{eq:shifted-conjugation-on-upper-triangular-coordinates} for $s = s_w$ takes the form 
\begin{equation}
\label{eq:easy-s-shifted-conjugation} 
\Ad_{\textrm{w}_0 w_\iota} (D_\iota) A_\iota \Ad_{\textrm{w}_0 w_{\iota \circ \varphi}}(D_{\iota \circ \varphi})^{-1} .
\end{equation}
In particular, the composite action
\[\begin{tikzcd}
	{T^{\mathcal{J}}} && {T^{\mathcal{J}}} & {\operatorname{Aut}(\tilde{\mathcal{A}})}
	\arrow["{\operatorname{Ad}_{\textrm{w}_0 w}^{-1}}", from=1-1, to=1-3]
	\arrow[from=1-3, to=1-4]
\end{tikzcd}\]
is now just the shifted conjugation action. This observation will be important when computing $T^\cJ$-invariant sections.

By inserting an appropriate shift, we can find coordinates on any chart $\tld{U}_s$ on which the action of $T^\cJ$ is ``nice'', as in the lemma below. However, note that these coordinates are no longer upper triangular.
\begin{lemma}
\label{translated-coordinates}
For any $s, t \in \underline{W}$, the map
\begin{align*}
\left(\tld{\cA}t\right)^{\nabla_{\tld{z}^*,s \omega}} & \lrisom \tld{U}_{ts} \\
(A_\iota)_{\iota \in \cJ} & \mapsto \left( I_1 \tld{z}_\iota^* A_\iota s_\iota v^{\omega_\iota} \right)_{\iota \in \cJ}
\end{align*}
is also an isomorphism, which is equivariant for the shifted conjugation action on $\tld{U}_{ts}$ and the action of $T^\cJ$ on $\tld{\cA}t$ given by \eqref{eq:shifted-conjugation-on-upper-triangular-coordinates}. In particular, this action is given by \eqref{eq:easy-s-shifted-conjugation} when $s=s_w$.
\end{lemma}
\section{Upper bounds on global sections}
\label{section:upper-bounds}
From now on we assume that $K = \Q_p$; the extension to the case $K$ unramified will be done in Section \ref{section:general-case}.
Let us recall and spell out how some of the notation from Section \ref{section:step0} simplifies in this case:
\begin{enumerate}
\item Any product indexed over $\cJ$ has a single factor, corresponding to the identity embedding $\Q_p \hookrightarrow \Q_p$.
\item The shifted conjugation action of $T^\cJ$ is just the conjugation action of $T$.
\end{enumerate}

Assume that we are in the situation of Theorem \ref{presentation-of-EG-component}. That is, $\sigma$ is a $(3n-1)$-generic Serre weight with a chosen lowest alcove presentation $(\tld{w},\omega) = (t_\nu w, \omega)$. Recall that from Definition \ref{*}, $$\tld{z}^* = (\textrm{w}_0 \tld{w})^* = \tld{w}^* \textrm{w}_0 = w^{-1} t_\nu \textrm{w}_0 = t_{w^{-1} \nu} w^{-1} \textrm{w}_0 . $$
In this section we will describe a candidate for the ring of global sections on $\cX(\sigma)$. What we will actually do is to describe an upper bound on the ring of $T$-invariant sections $\cO(\tld{\cC}_\sigma)^{T}$ on $\tld{\cC}_\sigma$, which by Corollary \ref{global-functions-as-invariant-sections} amounts to the same thing.

The main idea is rather simple. For each $s \in W$ we have an open $\tld{U}_s = \tld{S}^\circ(\tld{z}^*,st_\omega)^{\nabla_0} = I_1 \backslash I_1 \tld{z}^* I s v^\omega \cap \tld{\Fl}^{\nabla_0}$ inside $\tld{\cC}_\sigma$, as described in Section \ref{section:open-charts}. We also have a larger open $\tld{U} = \bigcup_{s \in W} \tld{U}_s \subset \tld{\cC}_\sigma$ which is irreducible by Lemma \ref{large-open}. It follows that for any choice of $s \in W$, restriction of functions yields $T$-equivariant embeddings
$$ \cO(\tld{\cC}_\sigma) \hookrightarrow \cO(\tld{U}) \hookrightarrow \cO(\tld{U}_s) .$$
By taking $T$-fixed points, we obtain embeddings
$$ \cO(\cX(\sigma)) \hookrightarrow \cO(\tld{U})^{T} \hookrightarrow \cO(\tld{U}_s)^{T} . $$
It turns out that the first inclusion is an isomorphism, but that is the subject of Section \ref{section:step3}.
The goal of this section is to give explicit descriptions of the upper bounds $\cO(\tld{U})^{T}$ and $\cO(\tld{U}_s)^T$ for a suitable choice of $s$.
In fact the suitable choice of $s$ is $s = s_w = \textrm{w}_0 w$ from Definition \ref{easy-s}.
This makes the computation of $\cO(\tld{U}_s)^{T}$ easy, and we call this ring the \emph{simple upper bound} for $\cO(\cX(\sigma))$.

To compute the \emph{optimal upper bound} $\cO(\tld{U})^T$, we first produce elements of this ring by explicit construction. The subring generated by these elements inside the larger ring $\cO(\tld{U}_{s_w})^T$, which by now we have computed, give a lower bound for $\cO(\tld{U})^T$.
Finally, we prove that this lower bound is sharp by showing that no other elements of $\cO(\tld{U}_s)^{T}$ extend to $T$-invariant functions on $\tld{U}_{t}$ for all $t \in W$.

\subsection{Simple upper bound} \label{Simple upper bound}
We now compute $\cO(\tld{U}_{s_w})^T$, where $s_w = \textrm{w}_0 w$ from Definition \ref{easy-s}.
\begin{lemma}
\label{simple-upper-bound}
Restriction of functions via the canonical map $T \to \tld{U}_{s_w}$ yields an isomorphism $\cO(\tld{U}_{s_w})^{T} \cong \cO(T)$.
\end{lemma}
\begin{proof}
Write $s = s_w$ in order to simplify notation. Let $ U \defeq \prod_{\iota \in \cJ} I \backslash I \tld{z}_\iota^* \cK v^{\omega_\iota} \cap \Fl^{\nabla_0}_\cJ. $
Let $\cA_s \subset L^+ U$ satisfy that $\tld{\cA}_s = T \cA_s \subset L^+B$ is the upper triangular coordinates on $\tld{U}_s$ from Lemma \ref{upper-triangular-coordinates}.
Then, note that via the isomorphism $\tld{\cA}_s \lrisom \tld{U}_s$, the action \eqref{eq:shifted-conjugation-on-upper-triangular-coordinates} on $\tld{A}_s$ which corresponds to the shifted conjugation action on $\tld{U}_s$, in our case $s = \textrm{w}_0 w$ simplifies to
$$ D \cdot A = \Ad_s (D) A \Ad_s(D)^{-1} \,\,\,\,\, \text{ for } (D,A) \in T \times \tld{A}_s . $$
So we can describe the action as follows
$$ T\xrightarrow[]{\Ad_s} T\xrightarrow{}\Aut(\tld{\cA_s}).$$
Since we are interested in the $T$-invariant sections only, we can disregard the $\Ad_s$.

In this case, the shifted conjugation action is trivial on $T$ and therefore $\cO(T)^{T} = \cO(T)$.
Furthermore, since $\cA_s$ lies inside $L^+U$, we get $\cO(\cA_s)^{T}\cong \F $.  
Combining these two computations and noting that we have a $T$-equivariant isomorphism $\tld{U}_s\cong T\times \cA_s$ we get
$$ \cO(\tld{U}_s)^{T} \cong \left( \cO(T) \otimes_\F \cO(\cA_s) \right)^{T} \cong \cO(T) \otimes_\F \cO(\cA_s)^{T} \cong \cO(T) \otimes_\F \F, $$
which is exactly what we wanted to show.
\end{proof}

\begin{rmk}
The above lemma is reproved for the case $K$ unramified in Lemma \ref{lemma:functions-on-open-general-case}.
In that case, the shifted conjugation action of $T^\cJ$ on itself is non-trivial, so the computation becomes a little more involved.
\end{rmk}

\begin{rmk}\label{globalsectionsw}
Let $c_1,c_2,\dots, c_n : T \to \bG_m$ be the diagonal coordinates on the torus. Then Lemma \ref{simple-upper-bound} shows that
$$ \cO(\tld{U}_{s_w}) \cong \F[c_1^{\pm 1}, c_2^{\pm 1}, \dots, c_n^{\pm 1}] . $$
These coordinates will be used later.
\end{rmk}

\begin{exam} \label{toy_case}
Let's give an explicit calculation in $n=3$, $w = \textrm{w}_0=(13)$, $\eta_w=(2,1,0)$. In this case, $\tilde w= w^{-1} v^{\eta_w} \textrm{w}_0=\mathrm{Diag}(1, v, v^2)$, and $\widetilde{U_1}=I_1 \backslash I_1 \tilde w I v^{\omega} \cap \widetilde{\Fl}^{\nabla_0}$. Note that $I_1 \backslash (I_1 \widetilde{w} I) \to I \backslash (I_1 \widetilde{w} I) = I \backslash (I \widetilde{w} I)$ is a trivial $T$-torsor, it suffices to understand $I \backslash (I \widetilde{w} I) v^ \omega  \cap \widetilde{\Fl}^{\nabla_0}$. We have an isomorphism $(I \cap \widetilde{w}^{-1} I \widetilde{w}) \backslash I \cong I \backslash (I \widetilde{w} I)$ via the map $A \mapsto B=\widetilde{w}A$. Using the notations from Definition \ref{defined-by-valuation-bounds}, we can calculate that

 \[
        (I \cap \widetilde{w}^{-1} I \widetilde{w}) = 
       \begin{pmatrix}
            [0,0] & [1,\infty] & [2,\infty] \\
           [1,\infty]& [0,0] & [1,\infty]\\
            [1,\infty]& [1,\infty]&[0,0]
        \end{pmatrix} 
    \]
    
    \[
        (I \cap \widetilde{w}^{-1} I \widetilde{w}) \backslash I = 
        \left\{  \begin{pmatrix}
             1 & \alpha_0 & \delta_0 + \delta_1v \\
             0 & 1 & \beta_0 \\
             0 & 0 & 1
        \end{pmatrix} \right\}
    \]
    
     Checking the monodromy condition on $B = \widetilde{w} A v^\omega$, where $A$ is a matrix as above, gives us
    \[
        \left(v \frac{dB}{dv} \right) B^{-1} \in \frac{1}{v} \mathrm{Lie}(I) \iff \delta_0(\omega_3 - \omega_1) = \alpha_0 \beta_0 (\omega_2 - \omega_1)
    \]
    
  Note that $\omega_1 \ne \omega_3$ when $\sigma$ is sufficiently generic, hence $ I \backslash (I \widetilde{w} I) v^ \omega  \cap \widetilde{\Fl}^{\nabla_0} \cong \mathbb{A}^3$ with global sections $\alpha_0, \beta_0, \delta_1$. Going back to $I_1$, we can see that $$I_1 \backslash (I_1 \widetilde{w} I) v^\omega =  
 \left\{      \widetilde{w}  \begin{pmatrix}
             c_1 & \alpha_0 & \delta_0 + \delta_1v \\
             0 & c_2 & \beta_0 \\
             0 & 0 & c_3
        \end{pmatrix}v^\omega \right\}
 $$

  So $\widetilde{U}_1 = I_1 \backslash (I_1 \widetilde{w} I) v^\omega  \cap \widetilde{\Fl}^{\nabla_0} \cong \mathrm{Spec}(\mathbb{F}[\alpha_0, \beta_0, \delta_1, c_1^{\pm 1}, c_2^{\pm 1}, c_3^{\pm 1}])$. After taking the $T$-invariant sections, the off-diagonals will vanish, and hence we get $\mathcal O(\widetilde{U}_1)^T=\mathbb{F}[c_1^{\pm 1}, c_2^{\pm 1}, c_3^{\pm 1}]$.
\end{exam}

\subsection{Invariant sections on auxiliary charts }
\label{section:auxiliary-charts}
In addition to the $T$-invariant sections on the chart $\tld{U}_{s_w}$ that we computed in Lemma \ref{simple-upper-bound}, we will need to know the $T$-invariant sections on certain other subschemes of $\tld{U}$ as well.
Specifically, let $s_k \in W$ be the element corresponding to the permutation that swaps $k$ and $k+1$, and let $\tld{\cA}$ be the upper triangular coordinates from Lemma \ref{upper-triangular-coordinates}.

Consider the subfunctor $\tld{\cB}_k \subset \tld{\cA} s_k$ defined by
\begin{align*}
\tld{\cB}_k \defeq
\left\lbrace
\begin{pmatrix}
d_1 & & & & & & & & \\
& d_2 & & & & & & & \\
& & \ddots & & & & & & \\
& & & d_{k-1} & & & & & \\
& & & & a & d_k & & & \\
& & & & d_{k+1} & & & & \\
& & & & & & d_{k+2} & & \\
& & & & & & & \ddots & \\
& & & & & & & & d_n
\end{pmatrix}
\middle|
d_1,d_2,\dots,d_n \in \bG_m, a \in \bG_a
\right\rbrace .
\end{align*}
In fact, $\tld{\cB}_k$ is contained in the locus satisfying the monodromy condition ${\nabla_{\tld{z}^*,s\omega}}$ from \eqref{advanced-monodromy-condition}.
This is because if $A \in \tld{\cB}_k(R)$, then
$$ v \left( \frac{d}{dv} A \right) A^{-1} + A \operatorname{diag}(s \omega) A^{-1} = A \operatorname{diag}(s \omega) A^{-1} $$
has non-zero terms at most 1 away from the diagonal, and all non-zero terms being constants, so this is contained in $\frac{1}{v} \Lie \Ad_{\tld{z}^{-*}} (I)$ by the description of $\Ad_{\tld{z}^{-*}}(I)$ in terms of valuation bounds from Lemma \ref{conjugated-iwahori-description}.
By Lemma \ref{translated-coordinates}, it follows that there is a closed embedding 
\begin{align*}
\tld{\cB}_k & \hookrightarrow \tld{U}_{s_k s_w} \\
A& \mapsto I_1 \tld{z}^* A s_w v^\omega
\end{align*}
which is equivariant for the action of $D \in T$ on $A \in \tld{\cB}_k$ given by $\Ad_{\textrm{w}_0 w}(D) A\Ad_{\textrm{w}_0 w}(D)^{-1}$ and the conjugation action of $T$ on $\tld{U}_{s_k s_w}$.
\begin{lemma}
\label{step1.5}
The $T$-invariant functions on $\tld{\cB}_k$ are given by $$\cO(\tld{\cB}_k)^T \cong \F[ d_1^{\pm 1}, d_2^{\pm 1}, \dots,  d_{k-1}^{\pm 1} , a , (d_k d_{k+1})^{\pm 1} , d_{k+2}^{\pm 1}, \dots, d_n^{\pm 1} ] .$$
\end{lemma}

\begin{proof}
    Let us give a sketch. For more details we refer to the proof of the more general statement in Lemma \ref{step1.5-general}. 

    We first note that $T$ acts trivially on all the functions $d_1^{\pm 1},...,d_{k-1}^{\pm 1}, d_{k+2}^{\pm 1},..., d_{n}^{\pm 1}$.
    So it remains to determine the $T$-invariant functions on 
    $\begin{pmatrix}
        a & d_k \\
        d_{k+1} & 0
    \end{pmatrix}.$
    We see that $T$ also acts trivially on $a$. 
    Let \(\cE_{i} \in X^*(T)\) denote the character \(\operatorname{diag}(t_{1},t_{2},\dots,t_{n}) \mapsto t_{i}\). 
    Then \(T\) acts on the function $d_k$ it via $\cE_{k}-\cE_{k+1}$ and on $d_{k+1}$ via $-(\cE_{k}-\cE_{k+1})$,
    thus the product of these two functions $d_k\cdot d_{k+1}$ and its inverse are invariant.
\end{proof}

\subsection{Optimal upper bound}
\label{section:step2}
Having computed the $T$-invariant sections on the chart $\tld{U}_{s_w}$  in Lemma \ref{simple-upper-bound} , we will now compute the $T$-invariant sections on $\tld{U}$.
We do so by first producing some $T$-invariant functions on $\tld{U}$, that in the end turn out to generate the entire ring $\cO(\tld{U})^T$.

Recall that $\cK = L^+ \GL_n$. Let $N \subset \GL_n$ denote the subgroup of lower triangular unipotent matrices. The canonical reduction modulo $v$ map $\pi_v : \cK \to \GL_n$ maps $\Ad_{\tld{z}^{-*}} I_1 \cap \cK$ to $N$ by Lemma \ref{conjugated-iwahori-description}. Therefore, the chain of maps
\begin{align*}
\alpha : \tld{U} \hookrightarrow I_1 \backslash I_1 \tld{z} \cK v^\omega \lrisom (\Ad_{\tld{z}^{-*}} I_1 \cap \cK) \backslash \cK  \stackrel{\pi_v}{\longrightarrow} N \backslash \GL_n \\
I_1 \tld{z}^* A v^\omega \mapsto (\Ad_{\tld{z}^{-*}} I_1 \cap \cK) A  \mapsto N \pi_v(A)
\end{align*}
gives a way to produce functions on $\tld{U}$, namely by pulling back functions on $N \backslash \GL_n$.
Note that the action of $T$ on $\GL_n$ which makes $\alpha$ equivariant, where $T$ acts on $\tld{U}$ by conjugation, is given by $\Ad_{\textrm{w}_0 w}(t) A t^{-1}$ for $(t,A) \in T \times \GL_n$. 
For the purposes of producing $T$-invariant functions on $N \backslash \GL_n$ with respect to this action, we use the substitution
$$A = B s_w = B \textrm{w}_0 w.\footnote{Formally, this corresponds to the right translation map $\cdot s_w^{-1} : N \backslash \GL_n \to N \backslash \GL_n$ sending $NA$ to $NB$. The resulting $T$-invariant functions on $\tld{U}$ are then produced by pulling back via $(\cdot s_w^{-1}) \circ \alpha$.}$$ 
Then the action of $T$ in terms of $B$ is given by
\begin{equation}
\label{period-action}
t.B=\Ad_{\textrm{w}_0w}(t) B \Ad_{\textrm{w}_0w}(t^{-1}) \,\,\,\,\, \text{for }(t,B) \in T \times \GL_n ,
\end{equation}
and $T$-invariant functions for this action are precisely the $T$-invariant functions for the ordinary conjugation action.
Such functions are easy to produce:
For $1 \leq i \leq n$ consider the function $x_i$ on $N \backslash \GL_n$ given by
\begin{align*}
x_i : N \backslash \GL_n & \to \bA^1 \\
N B & \mapsto \text{upper left }(i \times i)\text{-minor of }B .
\end{align*}
Then $x_i$ is invariant for the action of $T$ given by conjugation, therefore also invariant for the action of $T$ given by \eqref{period-action}.
\begin{rmk}
One can show that the inclusion $\F[x_1,x_2,\dots,x_{n-1},x_n^{\pm 1}] \subseteq \cO\left(N \backslash \GL_n \right)^T$ is an isomorphism, but we will not need this fact.
\end{rmk}
Using the map $\alpha$ above, we obtain $T$-invariant functions $\alpha^*x_i$ on $\tld{U}$, which in abuse of notation we also denote by $x_i$. Explicitly,
\begin{align*}
x_i : \tld{U} \subset I_1 \backslash I_1 \tld{z}^* \cK v^\omega & \to \bA^1 \\
I_1 \tld{z}^* (B s_w) v^\omega & \mapsto \text{upper left }(i \times i)\text{-minor of }B \mod v .
\end{align*}
Note the presence of $s_w$ on the left hand side; this is important in what follows.

We will now show that these functions $x_i$ generate the ring of all $T$-invariant functions on $\tld{U}$. Towards this end, we will need to describe the restriction of $x_i$ to the open $\tld{U}_{s_w} \subset \tld{U}$ from Section \ref{globalsectionsw} as well as to the locally closed $\tld{\cB}_k \subset \tld{U}$ from Section \ref{step1.5}.
Consider the following diagram:

\begin{tikzcd}
	&&  \cO\left(\left[\widetilde{U}_{s_w}/T\right]\right) \\
	 \cO\left( \left[N \backslash \GL_{n} /T \right] \right) &  \cO\left(\left[\widetilde{U}/T\right]\right) & \cO\left(\left[\widetilde{U}_{s_ks_w}/T\right]\right) & \cO\left(\left[\tld{\cB}_k/T\right]\right)
	\arrow["\alpha",from=2-1, to=2-2]
	\arrow[hook, from=2-2, to=1-3]
	\arrow[hook, from=2-2, to=2-3]
	\arrow[ from=2-3, to=2-4]
\end{tikzcd}

Let us compute the compositions of these maps:
\begin{itemize}
    \item Recall that $\tld{\cA}_{s_w}$ denotes the upper triangular coordinates on $\tld{U}_{s_w}$ from Lemma \ref{upper-triangular-coordinates2}, and $x_{i}|_{\tld{U}_{s_w}}$ as a function on $\tld{\cA}_{s_w}$ is the same as the upper $(i \times i)$-minor modulo $v$. Since $\tld{\cA}_{s_w}$ is upper triangular, this is the same as the product of the first $i$ diagonal entries modulo $v$. Hence the map $\cO\left( \left[N \backslash \GL_{n} /T \right] \right)\to \cO\left(\left[\widetilde{U}_{s_w}/T\right]\right)$, where the RHS by remark \ref{globalsectionsw} is isomorphic to $ \F[c_1^{\pm 1},...,c_n^{\pm 1}]$, sends
    $x_i$ to $c_1c_2\cdots c_i$,
    \item The bottom composition $r_{s_ws_k}\colon\cO\left( \left[N \backslash \GL_{n} /T \right] \right)\to \cO\left(\left[\tld{\cB}_k/T\right]\right)$, where the RHS \ref{step1.5} is isomorphic to $ \F[ d_1^{\pm 1}, d_2^{\pm 1}, \dots,  d_{k-1}^{\pm 1} , a , (d_k d_{k+1})^{\pm 1} , d_{k+2}^{\pm 1}, \dots, d_n^{\pm 1} ]$ , is given by
  \[   
r_{s_ws_k}(x_i)=
     \begin{cases}
       d_1 d_2 \cdots d_{i} &\quad\text{if } i<k,\\ 
       d_1 d_2 \cdots d_{k-1} a &\quad\text{if } i=k,\\ 
       - d_1 d_2 \cdots d_i &\quad\text{if } i>k.\\
     \end{cases}
\]
\end{itemize}
Now we have all the ingredients to prove the following lemma
\begin{lemma} 
\label{step2}
The functions on $\left[\widetilde{U}_{s_w}/T\right]$, as described in Remark \ref{globalsectionsw}, that extend to $\left[\widetilde{U}/T\right]$ are exactly generated by $c_{1}$, $c_{1}c_{2}$, ..., $c_{1}c_{2}...c_{n-1}$, $(c_{1}c_{2}...c_{n-1})^{\pm1}$. Therefore 
$$\cO\left(\left[\widetilde{U}/T\right]\right)\cong \F[x_1,...,x_{n-1},x_n^{\pm 1}].$$
\end{lemma}
\begin{proof}
We will write $x_i$ for section in $\cO\left( \left[N \backslash \GL_{n} /T \right] \right)$, and also identify it with its image inside $\cO\left(\left[\widetilde{U}/T\right]\right)$ and $\cO\left(\left[\widetilde{U}_{s_w}/T\right]\right)$. We will also regard the following equalities as being localized at $x_i$. That being said, we have
$$\cO\left(\left[\widetilde{U}_{s_w}/T\right]\right)\cong \F[c_1^{\pm 1},...,c_n^{\pm 1}]= \F[x_1^{\pm 1},...,x_n^{\pm 1}].$$

Since $\cO\left(\left[\widetilde{U}/T\right]\right)\hookrightarrow \cO\left(\left[\widetilde{U}_{s_w}/T\right]\right) $ is injective, any global section $\lambda\in \cO \left( \left[ \widetilde{U} / T \right] \right)$ can be expressed uniquely as \[
    \lambda = \sum_{\underline{i}=(i_1,...,i_n)\in\bZ^{n}} \lambda_{\underline{i}} x_{1}^{i_1}x_{2}^{i_2}...x_{n}^{i_n}.
\]
Let us fix a $k<n$. Then by the above computation \begin{equation*}
    r_{s_ws_k}(\lambda) = \sum_{i_{k}\in \bZ} a^{i_k}\sum_{\underline{i}= \\ (i_1,...,\hat{i}_{k},...i_n)\in\bZ^{n-1}}\pm \lambda_{\underline{i}} d_1^{i_1}...(d_1\cdots d_{k-1})^{i_{k-1}}(d_1\cdots d_{k-1})^{i_k}(d_1\cdots d_{k+1})^{i_{k+1}}...(d_1\cdots d_{n})^{i_n}.
\end{equation*}
Note that the $\pm$ sign comes from the definition of $r_{s_ws_k}$.
We see that for the function to extend, we can only allow non-negative powers of $a^{i_k}$. Therefore we get $\lambda_{\underline{i}}=0$ whenever $i_k<0$ for some $k\neq n$.

If $k=n$, then $x_n^{\pm 1}$ does extend to $\widetilde{U}$, since the $x_n$ is the determinant (up to a sign) and therefore always invertible.
Therefore we get the desired result:
$$\cO\left(\left[\widetilde{U}/T\right]\right)\cong \F[x_1,...,x_{n-1},x_n^{\pm 1}].$$
\end{proof}

\begin{exam}
We continue the example discussed in Example \ref{toy_case}. When $s=(12)$, an element in $ ((I_1 \cap \widetilde{w}^{-1} I_1 \widetilde{w}) \backslash I) s$ is of the form 
      $$  \begin{pmatrix}
             \alpha_0 & c_1 & \delta_0 + \delta_1v \\
             c_2v & 0 & \beta_0 \\
             0 & 0 & 1
        \end{pmatrix} $$

   Rewriting this matrix under the coordinates of $\widetilde{U}_1$ through left multiplication (i.e. making it upper-triangular by row operations), it becomes 
         $$  \begin{pmatrix}
             \alpha_0 & c_1 & \delta_0 + \delta_1v \\
            0 &  -c_1c_2v/\alpha_0 & (\beta-\delta_0c_2/\alpha_0)v \\
             0 & 0 & c_3v^2
        \end{pmatrix} $$
        Imposing the monodromy condition and taking $T$-invariants, we conclude that $\mathcal O(\widetilde U_s)^T=\mathbb F[\alpha_0, -c_1c_2, c_3]$, with $c_1,c_2,c_3$ being invertible. So we can see that this cuts off the possibility of the functions on $\widetilde U$ with $\alpha_0$ also being invertible. Doing the same for $s=(23)$ and $s=(13)$ will cut off more possibilities in Example \ref{toy_case} and finally yield our desired optimal upper bound $\mathbb F[x_1,x_2,x_3^{\pm1}]$.
            \end{exam}

\section{Extendability of functions on \texorpdfstring{$\widetilde U$}{U'}}
\label{section:step3}
This section is a continuation of the previous one; as such, we keep the same assumptions and notation from that section.
Having described the ring of $T$-invariant functions $\cO(\tld{U})^T$ on the open $\tld{U} \subset \tld{\cC}_\sigma$, we will prove in this section that all such functions extend to $\cO(\tld{\cC}_\sigma)^T$, or in other words that the restriction $\cO(\tld{\cC}_\sigma)^T \lrisom \cO(\tld{U})^T$ is an isomorphism.

We will reduce the problem of extending functions to the geometry of Schubert cells, and we refer to Appendix \ref{demazure-resolutions} for results and proofs about (intermediate) Demazure resolutions in what follows. 
Here, we will give an overview of the important results proved therein and how to apply them to our situation to get the desired result. 
We stress that the appendix is using different conventions than the rest of the paper: 
Instead of the affine flag variety $I\setminus LG$, we work with the ``transpose'' affine flag variety $LG/\bar I$, as this simplifies the combinatorics.
More precisely, we get a condition in Proposition \ref{extending-functions-prop}, which in our case is easy to verify via an argument using alcove walks (see Theorem \ref{important}). 
In the final corollary of the appendix we translate the result back to our set-up here.  

\begin{thm}
	\label{main-theorem-base-case}
	The functions $x_1, x_2, \cdots, x_n, x_n^{-1}$ can be extended to $T$-invariant global functions on $\widetilde{\mathcal{C}}_\sigma$.
	\end{thm}
	\begin{proof}
	
	\begin{enumerate}
		\item We first reduce this problem to Schubert varieties. Consider the following commutative diagram  
		\[\begin{tikzcd}
			{\tld{U}_{s_w}} &\widetilde{\mathcal{C}}_\sigma \\
			{\widetilde{S}^\circ(\tld{z}^*,s_wt_\omega)} & {\widetilde{S}(\tld{z}^*,s_wt_\omega).}
			\arrow[hook, from=1-1, to=1-2]
			\arrow[hook, from=1-2, to=2-2]
			\arrow[hook, from=1-1, to=2-1]
			\arrow[hook, from=2-1, to=2-2]
		\end{tikzcd}\]
		Here, the horizontal arrows are open embeddings and the vertical arrows are closed immersions, induced by intersection with the monodromy condition.\footnote{The inclusion $\tld{\cC}_\sigma = \tld{S}^{\nabla_0}(\tld{z}^*,t_\omega) \subset \tld{S}(\tld{z}^*,t_\omega)^{\nabla_0}$ is typically not an equality, as mentioned in Remark \ref{schubert-cell-discrepancy}. Nevertheless, the vertical maps are closed immersions, which suffices for our argument.}
		Given a $T$-invariant function $f$ on $\tld{U}_{s_w}$, there is an obvious choice for a lift $f_{\widetilde{S}^\circ(\tld{z}^*,s_w t_\omega)}$.
		It suffices to extend the function $f_{\widetilde{S}^\circ(\tld{z}^*,s_w t_\omega)}$ to a function $f_{\widetilde{S}(\tld{z}^*,s_w t_\omega)}$ on the whole space $\widetilde{S}(\tld{z}^*,s_w t_\omega)$. Then the restriction of $f_{\widetilde{S}(\tld{z}^*,t_\omega)}$ to $\widetilde{\mathcal{C}}_\sigma$ will be our desired element. Note that the right shift by $s_w t_\omega$ is irrelevant for this question, so from now on we will drop it.

		\item Now we reduce the problem to extending functions along certain strata in the Schubert cell. First, according to \cite[Theorem 0.3]{pappas-rapoport} $\tld{S}(\tld{z}^*)$ is normal. Therefore, it suffices to construct global functions on open subschemes of codimension $\geq 2$ by the algebraic version of the Hartog's lemma (see \cite[Tag 031T]{stacks-project}).
		We can use the stratification 
		$$ \widetilde{S}(\tld{z}^*)=\widetilde{S}^\circ(\tld{z}^*)\cup \bigcup_{\substack{\tld{z}'\in \cI}} \widetilde{S}^\circ(\tld{z}') \cup \bigcup_{\substack{\tld{z}'\in \cI'}} \widetilde{S}^\circ(\tld{z}'), $$
		where $\cI=\{\tld{z}' | \tld{z}'<\tld{z}^*\ , \   l(\tld{z}')=l(\tld{z}^*)-1 \}, \cI'=\{\tld{z}' | \tld{z}'<\tld{z}^* \ , \  l(\tld{z}')<l(\tld{z}^*)-1 \}$. Note that $\dim_\bF \widetilde{S}^\circ(\tld{z})=l(\tld z) + n$.
		Thus in order to extend $f_{\widetilde{S}^\circ(\tld{z}^*)}$ to $\widetilde{S}(\tld{z}^*)$, it suffices to show that the function extends across the codimension 1 strata $\widetilde{S}^\circ(\tld{z}')$, where $\tld{z}' \in \cI$.
		
		\item To extend functions along the codimension one strata, we use Demazure resolutions $\tld{D}^\circ(\tld{z})$ and ``intermediate'' Demazure resolutions $\tld{D}^\circ(\tld{z}^*,i)$ which correspond to $\tld{S}^\circ(\tld{z}^*) \cup \tld{S}^\circ(\tld{z}')$ (see Lemma \ref{pullback}).
		We will use Lemma \ref{extend} to reduce to the question of when a global section extends along the open embedding 
		\[\begin{tikzcd}
			{\widetilde{D}^\circ(\tld{z}^*)} & {\widetilde{D}^\circ(\tld{z}^*,i)} \\
			{\bA^1.}
			\arrow[hook, from=1-1, to=1-2]
			\arrow[ from=1-1, to=2-1]
			\arrow[dashed, from=1-2, to=2-1]
		\end{tikzcd}.\]
		This yields a general criterion stated in Proposition \ref{extending-functions-prop} (note that the open Demazure resolution is isomorphic to the open Schubert cell). 
		
		\item The functions  $x_j$ on the open Schubert cell $I_1 \backslash I_1 \tld{z}^* I$ that we want to extend are given by 
		\begin{align*}
		x_j : I_1 \backslash I_1 \tld{z}^* I & \to \bA^1 \\
		I_1 \tld{z}^* A& \mapsto \text{upper left }(j \times j)\text{-minor of }A \mod v .
		\end{align*}
		That these can be extended is exactly the content of Corollary \ref{extending-minors-in-general} (note again that we use the ``transpose" convention in the appendix, hence the difference in the notation for the function $x_j$). To be more precise, the core computation that the general extension criterion (Proposition \ref{extending-functions-prop}) is fulfilled for the $x_j$'s is the Theorem \ref{important}. The corollary afterwards is just translating back the conventions to the ones we use in the rest of the paper.

		\item We have seen that $x_1,...,x_n$ extend to $\widetilde{S}(\tld{z}^*)$, hence we can (after translation by $s_w t_\omega)$ restrict them to $\widetilde{\mathcal{C}}_\sigma$.  Since $x_n: \widetilde{\mathcal{C}}_\sigma\to \bA^1$ factors through $\bG_m\subset \bA^1$, we can also extend $x_n^{-1}$. 
		Finally, all the functions are $T$-invariant on $\widetilde{\mathcal{C}}_\sigma$, since they are $T$-invariant on a dense open, namely $\tld{U}_{s_w}$.
	\end{enumerate}

	\end{proof}

\begin{exam}
	We continue the example discussed in Example \ref{toy_case}. Thus $\tld{z}^*=\mathrm{Diag}(1, v, v^2)$ and we know that $\mathcal I=\{\tld{z}'_1,\tld{z}'_2,\tld{z}'_3 \}$, where $$\tld{z}'_1=\begin{pmatrix}
		0 & 0 & v\\
		0 & v & 0\\
		v & 0 & 0
				\end{pmatrix}, \tld{z}'_2=\begin{pmatrix}
		0 & 1 & 0\\
		v &  0 & 0\\
		0  & 0 & v^2
				\end{pmatrix}, \tld{z}'_3=\begin{pmatrix}
		1 & 0 & 0\\
		0 &  0 & v\\
		0 & v^2 & 0
				\end{pmatrix} $$
		We draw them on the following \hyperref[picture]{picture}.
		\newline
		
		\begin{figure}[h!]
		\centering
		 \begin{tikzpicture}[scale=6]  \AutoSizeWeightLatticefalse
		\begin{rootSystem}{G}
		\setlength{\weightRadius}{1pt}
		\weightLattice{2}
		\roots[brown]
		\wt[root]{0}{0}
		\node[below left=1pt] at (hex cs:x=-0.16,y=0.8){\scalebox{0.55}{$
		 \begin{pmatrix}
		v & 0 & 0\\
		0 & v & 0\\
		0 & 0 & v\\
		\end{pmatrix}$}};
		\node[above right=1pt] at (hex cs:x=0.1,y=-0.8){\scalebox{0.55}{$
		\tilde z_1'=\begin{pmatrix}
		0 & 0 & v\\
		0 & v & 0\\
		v & 0 & 0
		\end{pmatrix}
		$}};
		\node[above right=1pt] at (hex cs:x=-0.95,y=-0.8){\scalebox{0.55}{$
		\tilde z_3' =\begin{pmatrix}
		1 & 0 & 0\\
		0 &  0 & v\\
		0 & v^2 & 0
		\end{pmatrix}$}};
		
		\node[above right=1pt] at (hex cs:x=1.05,y=-0.8){\scalebox{0.55}{$
		\tilde z_2'= \begin{pmatrix}
		0 & 1 & 0\\
		v & 0 & 0\\
		0 & 0 & v^2
		\end{pmatrix}$}};
		
		\node[below left=1pt] at (hex cs:x=0.93,y=-1.2){\scalebox{0.55}{$
		\tilde z^*=\begin{pmatrix}
		1 & 0 & 0\\
		0 & v & 0\\
		0 & 0 & v^2
		\end{pmatrix}$}};
		\end{rootSystem}
		\end{tikzpicture}
		\label{picture}
		\caption{A partial visualization of affine Weyl group of $\GL_3$. More precisely we can think of the extended affine Weyl group $\tilde{W}\cong W_{a}\times \Z$ as
		copies of $W_{a}$, the extended affine Weyl group of $\SL_3$, indexed by $\Z$. If we fix an $\ell\in\Z$ (in our figure $\ell=3$, because $\det(\tilde z^*)=v^3$ has order $3$), we can use the isomorphism between $W_{a}$
		and the alcoves of the standard apartment of $\SL_3$ to visualize the Bruhat order. Elements which are smaller and of length $l(\tilde z^*)-1$ correspond to the
		alcoves which are obtained by exactly one reflection along the walls corresponding to the omitted element in the reduced expression.}
		\end{figure}
\end{exam}

\section{The general case}
\label{section:general-case}

We now explain how to extend the previous calculations to the case of an unramified extension $K = \Q_{p^f}$ where $f > 1$.
Recall the setup in Section \ref{section:presentation-of-the-irreducible-components}.
Namely, $\sigma$ denotes a $(3n-1)$-deep Serre weight for $\GL_n(\cO_K)$ with lowest alcove presentation $(\widetilde{w},\omega)$ and $\tld{z}^* = (\textrm{w}_0 \tld{w})^*$.
There exists a closed irreducible subscheme $\widetilde{\cC}_\sigma \subset \Fl_\cJ^{\nabla_0}$ on which the torus $T^\cJ$ acts by shifted conjugation, and for which we have an isomorphism $\cX(\sigma) \cong \left[ \widetilde{\cC}_\sigma / T^\cJ \right]$, where the left hand side is the irreducible component of the Emerton-Gee stack labeled by $\sigma$. 
The goal of this section is to prove the following theorem.
\begin{thm}
\label{main-theorem}
We have an isomorphism $$\cO(\tld{\cC}_\sigma)^{T^\cJ} = \F[x_1,x_2,\dots,x_{n-1},x_{n}^{\pm 1}] ,$$ 
where the right hand side denotes the free polynomial ring in $n$ variables with one variable inverted.
\end{thm}
Concrete interpretations of the functions $x_i$ on open dense subschemes of $\tld{\cC}_\sigma$ is provided in Equations \eqref{eq:defn-of-xs} and \eqref{eq:defn-of-x}.
The rest of this section is dedicated to the proof of Theorem \ref{main-theorem}.

The strategy is exactly as before. Rather than repeating the idea we will content ourselves with making a few remarks about the similarities and differences between the base case $f=1$ and the general case $f>1$ under consideration. 
Note that the proof of Lemma \ref{lemma:functions-on-open-general-case}, which computes $\cO(\tld{U}_{s_w})^{T^\cJ}$, is more involved than the proof of Lemma \ref{simple-upper-bound}, which does the same computation in the case $f=1$. Fundamentally, the reason is that the action of $T^\cJ$ on itself by shifted conjugation is non-trivial, whereas the conjugation action of $T$ on itself is trivial. Similarly, the proof of Lemma \ref{step1.5-general} is more difficult than the proof of Lemma \ref{step1.5}. In the end, the computation of $T^\cJ$-invariant functions on $\tld{U}$ is done exactly as before.
The question of extending functions to $\tld{\cC}_\sigma$ reduces to the case $f=1$.

We now make some preparations needed for the proofs.
Let $\iota_0 : k \hookrightarrow \F$ be an embedding, from which we obtain an enumeration $\lbrace \iota_0,\iota_0 \circ \varphi^{-1},\dots,\iota_0 \circ \varphi^{-f+1} \rbrace$ of all embeddings $k \hookrightarrow \F$. By abuse of notations, the $j$-th embedding $\iota_0 \circ \varphi^{-j}$ will also be denoted as $j:k \hookrightarrow \bF$.
Then we can write
$$ \tld{U} = \prod_{j=0}^{f-1} \tld{U}_j = \prod_{j=0}^{f-1} I_1 \backslash I_1 \tld{z}^*_j \cK v^{\omega_j}, \,\,\,\,\, \text{ where }\cK = L^+ \GL_n .$$
For any $s \in \underline{W}$, we can also write
$$ \widetilde{U}_s = \prod_{j=0}^{f-1} \widetilde{U}_{s,j} , \,\,\,\,\, \text{ where } \widetilde{U}_{s,j} = \tld{S}^\circ(\tld{z}^*_j,s_jt_{\omega_j})^{\nabla_0} = I_1 \backslash I_1 \tld{z}^*_j I s_j v^{\omega_j} \cap \widetilde{\Fl}^{\nabla_0} . $$
Recall that these open charts were studied in Section \ref{section:open-charts}.

As we did in Section \ref{section:step2}, we can name certain candidate functions on $\tld{U}$.
Recall $s_w \in \underline{W} \cong \prod_{j=0}^{f-1} W$ from Definition \ref{easy-s}, namely $(s_w)_j = \textrm{w}_0 w_{j-1}$.\footnote{Note the difference between $\textrm{w}_0$, the longest Weyl group element, and $w_0$ the $0$'th component of $w$.}
For $1 \leq i \leq n$ and $0 \leq j \leq f-1$, let $x_{j,i}$ be the function
\begin{align}
x_{j,i} : \tld{U} & \to \bA^1 \label{eq:defn-of-xs} \\
I_1 \tld{z}^* A s_w v^\omega & \mapsto \text{upper left }(i \times i)\text{-minor of }A_j \mod v , \nonumber
\end{align}
where we use the shorthand notation
$$ I_1 \tld{z}^* A s_w v^\omega = \left( I_1 \tld{z}^*_0 (A_0 s_{w,0}) v^{\omega_0}, I_1 \tld{z}^*_1 (A_1 s_{w,1}) v^{\omega_1}, \dots , I_1 \tld{z}^*_{f-1} (A_{f-1} s_{w,f-1}) v^{\omega_{f-1}} \right) . $$
The function $x_{j,i}$ is well-defined by the same argument as in Section \ref{section:step2}.

It is very important to observe the presence of $s_w$ on the left hand side of \eqref{eq:defn-of-xs}.
As a consequence, if $\tld{\cA}_{s_w}$ denotes the upper triangular coordinates on $\tld{U}_{s_w}$ from Lemma \ref{upper-triangular-coordinates2}, then $x_{j,i}|_{\tld{U}_{s_w}}$ as a function on $\tld{\cA}_{s_w}$ has the same description as the upper $(i \times i)$-minor of the $j$-th component modulo $v$.
Note that since $\tld{\cA}_{s_w}$ is upper triangular by definition, this is the same as the product of the first $i$ diagonal entries of the $j$-th component modulo $v$.

For $1 \leq i \leq n$, let
\begin{equation}
\label{eq:defn-of-x}
x_i \defeq x_{0,i} x_{1,i} \cdots x_{f-1,i} .
\end{equation}
\begin{lemma}
\label{lemma:functions-on-open-general-case}
Let $s = s_w \in \underline{W}$ be the element defined in Definition \ref{easy-s}, i.e. $s_j = \textrm{w}_0 w_{j-1}$.
Then $\cO(\widetilde{U}_s)^{T^\cJ} = \F[x_1^{\pm 1},x_2^{\pm 1},\dots,x_{n}^{\pm 1}]$.
\end{lemma}
\begin{proof}
Let $\tld{\cA}$ and $\tld{\cA}_s$ be the upper triangular coordinates from Lemma \ref{upper-triangular-coordinates} and Lemma \ref{upper-triangular-coordinates2}, respectively.
Then $T^\cJ \subset \tld{\cA}_s \subset \tld{\cA} \subset T^\cJ L^+ U$ are closed subschemes, so it is sufficient to prove that $\cO(T^\cJ L^+U)^{T^\cJ} \cong \cO(T^\cJ)^{T^\cJ} \cong \F[x_1^{\pm 1} , x_2^{\pm 1}, \dots, x_n^{\pm 1}]$.
We will do this by explicit computation.

The shifted conjugation action of $T^\cJ$ on $\tld{U}_s$ corresponds to an action on $\tld{\cA}_s$ described by \eqref{eq:shifted-conjugation-on-upper-triangular-coordinates}, and the action on $T^\cJ L^+U^\cJ$ is given by the same formula. For our chosen $s$, this action is given by
\begin{align*} (D \cdot A)_j & = \Ad_{\textrm{w}_0 w_j} (D_j) A_j \Ad_{s_j}(D_{j-1})^{-1} & \\
& = \Ad_{\textrm{w}_0 w_j} (D_j) A_j \Ad_{\textrm{w}_0 w_{j-1}}(D_{j-1})^{-1} & \text{ for } (D,A) \in T^\cJ \times T^\cJ L^+ U^\cJ .
\end{align*}
In other words, up to the automorphism $\Ad_{\textrm{w}_0 w}$ of $T^\cJ$, the action is literally given by the shifted conjugation action 
\begin{equation}
\label{literally-shifted-conjugation}
 (D \cdot A)_j = D_j A_j D_{j-1}^{-1} \,\,\,\,\, \text{ for } (D,A) \in T^\cJ \times T^\cJ L^+U^\cJ .
\end{equation} 
Because the weight space for the action of $T^\cJ$ on $\cO(T^\cJ L^+ U)$ corresponding to weight 0 is insensitive to such an automorphism, we may and will assume that the action is given by \eqref{literally-shifted-conjugation}.

Note that $$ \cO(T^{\cJ} L^+ U^\cJ) = \F[c_{j,i}^{\pm 1}, a_{j,k,l,m} | 0 \leq j \leq f-1, 1 \leq i \leq n,  1 \leq k < l \leq n, m \in \Z_{\geq 0} ] , $$
where
\begin{align*}
c_{j,i}(A_0,A_1,\dots,A_{f-1}) & \defeq \text{the }(i,i)\text{'th entry of }A_j; \\
a_{j,k,l,m}(A_0,A_1,\dots,A_{f-1}) & \defeq \text{the }v^m\text{'th coefficient of the }(k,l)\text{'th entry of }A_j.
\end{align*}
The shifted conjugation action of $D = (D_0,D_1,\dots,D_{f-1}) \in T^\cJ$ on these functions is given by
\begin{align*}
D \cdot c_{j,i} & = D_{j,i}D_{j-1,i}^{-1} c_{j,i} = \chi_{j,i,i}(D) c_{j,i} ; \\
D \cdot a_{j,k,l,m} & = D_{j,k} D_{j-1,l}^{-1} a_{j,k,l,m} = \chi_{j,k,l}(D) a_{j,k,l,m} ,
\end{align*}
where $\chi_{j,k,l}(D) \defeq D_{j,k} D_{j-1,l}^{-1}$.

The monomials
\begin{equation}
\label{monomial}
\prod_{j,i,k,l,m} c_{j,i}^{m_{j,i}} a_{j,k,l,m}^{n_{j,k,l,m}}
\end{equation}
with $m_{j,i} \in \Z$ and $n_{j,k,l,m} \in \Z_{\geq 0}$ span the ring of functions and are eigenvectors for the shifted conjugation action. The weight of \eqref{monomial} is precisely
\begin{equation}
\label{weight}
\sum_{j,i,k,l,m} m_{j,i} \chi_{j,i,i} + n_{j,k,l,m} \chi_{j,k,l} .
\end{equation}
The statement we want to prove is equivalent to the assertion that if the monomial \ref{monomial} corresponds to weight 0, then $m_{0,i} = m_{1,i} = \cdots = m_{f-1,i}$ for all $i$ and all $n_{j,k,l,m} = 0$.

For $0 \leq j \leq f-1$ and $1 \leq i \leq n$ let $\cE_{j,i}(D) = D_{j,i}$ be the character of $X^*(T^\cJ)$ which picks out the $i$'th entry of the $j$'th factor. Then the $\cE_{j,i}$ is a basis for $X^*(T^\cJ)$ and there is a dual basis $\cE_{j,i}^\vee$ for $X_*(T^\cJ)$. Let
$$ \eta = (n-1) \sum_{j=0}^{f-1} \cE_{j,1}^\vee + (n-2) \sum_{j=0}^{f-1} \cE_{j,2}^\vee + \cdots + \sum_{j=0}^{f-1} \cE_{j,n-1}^\vee .$$
Then 
$$\langle \chi_{j,k,l} , \eta \rangle = \langle \cE_{j,k} - \cE_{j-1,l},\eta \rangle = l-k ,$$
which we note is positive for $k < l$ and zero for $k=l=i$. 

Now let $\chi$ denote the character \eqref{weight} and assume $\chi = 0$. Then in particular, $\langle \chi, \eta \rangle = 0$, which implies that all $n_{j,k,l,m} = 0$ by positivity. Hence
$$ 0 = \chi = \sum_{j,i} m_{j,i} \chi_{j,i,i} = \sum_{j,i} m_{j,i}(\cE_{j,i} - \cE_{j-1,i}) = \sum_{j,i} (m_{j,i} - m_{j+1,i}) \cE_{j,i} . $$
It follows that $m_{0,i} = m_{1,i} = \cdots = m_{f-1,i}$ for all $i$, which completes the proof.
\end{proof}
By construction each $x_i$ is a regular function defined on $\tld{U}$.
But as in Section \ref{section:step2}, we will show that $x_i^{-1}$ does not extend to $\tld{U}$ unless $i = n$, using the same method of proof which involved certain auxiliary charts $\tld{\cB}_k$.
These are defined similarly to how they were defined in Section \ref{section:auxiliary-charts}; in fact, they are defined exactly the same way in each coordinate.
In abuse of notation, let $s_k = (s_k, s_k, \dots, s_k) \in \underline{W}$ be the element which in each coordinate $j$ swaps $k$ and $k+1$.
Define the subfunctor $\tld{\cB}_k \subset \tld{\cA} s_k$, where $\tld{\cA}$ are the upper triangular coordinates of Lemma \ref{upper-triangular-coordinates}, by
\begin{align*}
\left( \tld{\cB}_k \right)_j \defeq
\left\lbrace
\begin{pmatrix}
d_{j,1} & & & & & & & & \\
& d_{j,2} & & & & & & & \\
& & \ddots & & & & & & \\
& & & d_{j,k-1} & & & & & \\
& & & & a_j & d_{j,k} & & & \\
& & & & d_{j,k+1} & & & & \\
& & & & & & d_{j,k+2} & & \\
& & & & & & & \ddots & \\
& & & & & & & & d_{j,n}
\end{pmatrix}
\middle|
d_{j,1},d_{j,2},\dots,d_{j,n} \in \bG_m, a_j \in \bG_a
\right\rbrace .
\end{align*}
By the same arguments as in Section \ref{section:auxiliary-charts} the map
\begin{align*}
\tld{\cB}_k & \hookrightarrow \tld{U}_{s_k s_w} \\
A & \mapsto I_1 \tld{z}^* A s_w v^\omega
\end{align*}
is a well-defined closed embedding which is equivariant for the action of $D = (D_0,D_1,\dots,D_{f-1}) \in T^\cJ$ on $A = (A_0,A_1,\dots,A_{f-1}) \in \tld{\cB}_k$ given by $(D \cdot A)_j = \Ad_{\textrm{w}_0 w_j}(D_j) A_j \Ad_{\textrm{w}_0 w_{j-1}}(D_{j-1})^{-1}$ and the shifted conjugation on $\tld{U}_{s_k s_w}$.
As before, consider the functions on $\tld{\cB}_k$ given by
\begin{align*}
d_i & \defeq d_{0,i} d_{1,i} \cdots d_{f-1,i} , \\
a & \defeq a_0 a_1 \cdots a_{f-1} ,
\end{align*}
for all $1 \leq i \leq n$. 
    
In the case that $\# \cJ$ is even, we will also need the functions
\begin{align*}
d_i^e & \defeq d_{0,i} d_{2,i} \cdots d_{f-2,i} , \\
d_i^o & \defeq d_{1,i} d_{3,i} \cdots d_{f-1,i} 
\end{align*}
for $i = k,k+1$.

The following result extends Lemma \ref{step1.5}.
\begin{lemma}
\label{step1.5-general}
We have 
$$ \cO(\tld{\cB}_k)^{T^\cJ} =
\begin{cases} 
\F \left[ d_1^{\pm 1}, d_2^{\pm 1}, \dots , a, (d_k d_{k+1})^{\pm 1}, d_{k+2}^{\pm 1},\dots,d_n^{\pm 1} \right], & \text{ if }\#\cJ \text{ is odd}, \\
\F \left[ d_1^{\pm 1}, d_2^{\pm 1}, \dots , a, (d_k^e d_{k+1}^o)^{\pm 1}, (d_k^o d_{k+1}^e)^{\pm 1}, d_{k+2}^{\pm 1},\dots,d_n^{\pm 1} \right], & \text{ if } \# \cJ \text{ is even}.
\end{cases}$$
\end{lemma}
\begin{proof}
The proof is similar to the previous one.
As in that proof, we can reduce to the case where $T^\cJ$ acts by literal shifted conjugation.
Note that 
$$\cO(\tld{\cB}_k) = \F \left[ d_{j,i}^{\pm 1}, a_j | 0 \leq j \leq f-1, 1 \leq i \leq n \right].$$
The weight spaces for the shifted conjugation action of $T^\cJ$ is spanned by monomials
$$ \prod_{j,i} d_{j,i}^{m_{j,i}} a_j^{n_j}  \,\,\,\,\, \text{ where }m_{j,i} \in \bZ, \,\, n_j \in \bZ_{\geq 0} . $$
The weight associated with such a monomial is precisely
\begin{equation}
\label{auxiliary-monomial}
 \sum_{j,i\neq k,k+1} m_{j,i} \chi_{j,i,i} + \sum_{j} \left( n_j \chi_{j,k,k} + m_{j,k} \chi_{j,k,k+1} + m_{j,k+1} \chi_{j,k+1,k} \right) ,
\end{equation} 
with $\chi_{j,k,l}(D) = (\cE_{j,k} - \cE_{j-1,l})(D) = D_{j,k}D_{j-1,l}^{-1}$ as in the last proof.
The assertion we want to prove is now that \eqref{auxiliary-monomial} is zero if and only if the following conditions are satisfied:
\begin{enumerate}
\item $m_{0,i} = m_{1,i} = \cdots = m_{f-1,i}$ for all $i \neq k, k+1$.
\item $n_{0} = n_{1} = \cdots = n_{f-1}$.
\item If $\# \cJ$ is odd, $m_{j,k} = m_{j',k} = m_{j'',k+1} = m_{j''',k+1}$ for all $j,j',j'',j'''$. If $\# \cJ$ is even, $m_{j,k} = m_{j',k} = m_{j'',k+1} = m_{j''',k+1}$ for even $j,j'$ and odd $j'', j'''$, and similarly for odd $j,j'$ and even $j'',j'''$.
\end{enumerate}
Denote \eqref{auxiliary-monomial} by $\chi$ and assume that $\chi = 0$. Then
\begin{equation}
\label{weight-of-auxiliary-monomial}
0 = \sum_{j,i\neq k,k+1} (m_{j,i} - m_{j+1,i})\cE_{j,i} + \sum_j \left( (n_j - n_{j+1} + m_{j,k} - m_{j+1,k+1} ) \cE_{j,k} + (m_{j,k+1} - m_{j+1,k}) \cE_{j,k+1} \right) ,
\end{equation}
which immediately gives (1). For the remaining part, note that \eqref{weight-of-auxiliary-monomial} gives the following two identites for all $j$:
$$ m_{j,k+1} = m_{j+1,k} \,\,\,\,\, \text{ and } \,\,\,\,\, m_{j,k} = m_{j+1,k+1} + n_{j+1} - n_j . $$
Therefore,
\begin{align*} 
 m_{j,k+1} & = m_{j+1,k} \\
& =  m_{j+2,k+1} + n_{j+1} - n_j \\
& = m_{j+3,k} + n_{j+1} - n_j \\
& = m_{j+4,k+1} + 2(n_{j+1} - n_j)  = \cdots \\
& = m_{j + 2N,k+1} + N(n_{j+1} - n_j)  .
\end{align*}
Since we have $m_{j+2N,k+1} = m_{j,k+1}$ for some $N$, we obtain (2). But then the above chain of equalities is just
$$ m_{j,k+1} = m_{j+1,k} = \cdots = m_{j+2N-1,k} = m_{j+2N,k+1} , $$
which implies (3). Conversely, it is easy to see that \eqref{weight-of-auxiliary-monomial} is satisfied upon assuming (1), (2), and (3).
\end{proof}
Now we have the ingredients needed to show the following extension of Lemma \ref{step2}.
\begin{lemma}
Restriction along the inclusion $\tld{U}_{s_w} \hookrightarrow \tld{U}$ identifies $\cO(\tld{U})^{T^\cJ}$ with the subring 
$$\F \left[ x_1,x_2,\dots,x_{n-1},x_n^{\pm 1} \right] \subset \F \left[ x_1^{\pm 1},\dots,x_n^{\pm 1} \right] = \cO(\tld{U}_{s_w})^{T^\cJ}, $$
with the latter equality coming from Lemma \ref{lemma:functions-on-open-general-case}.
\end{lemma}
\begin{proof}
Exactly like the proof of Lemma \ref{step2}.
\end{proof}
Because $\tld{U} \subset \tld{C}_\sigma$ is a dense open by the results of Section \ref{section:open-charts}, $\cO(\tld{\cC}_\sigma)^{T^\cJ} \hookrightarrow \cO(\tld{U})^{T^\cJ}$ embeds as a subring.
Therefore, Theorem \ref{main-theorem} now follows from the following result.
\begin{lemma}
Each of the functions $x_i$ from \eqref{eq:defn-of-x} extends to $\tld{C}_\sigma$.
\end{lemma}
\begin{proof}
Since $x_i = x_{0,i} x_{1,i} \cdots x_{f-1,i}$ by definition, it suffices to show that each $x_{j,i}$ extends to $\tld{C}_\sigma$.
Note that $x_{j,i}$ factors as $\tld{U}_{s_w} \to (\tld{U}_{s_w})_j \to \bA^1$, where the first map is the projection to the $j$'th factor, and the second map is the minor map $$y_i : I_1 \tld{z}^*_j A (s_w)_j v^{\omega_j} \mapsto \text{upper left }(i \times i)\text{-minor of }A \mod v .$$
As $\tld{z}^*_j = (\textrm{w}_0 \tld{w}_j)^*$, where $\textrm{w}_0 \tld{w}_j$ is anti-dominant by construction, it follows from Corollary \ref{extending-minors-in-general} that $y_i$ extends to the closure $\tld{\cC}_j$ of $(\tld{U}_{s_w})_j$ inside $\Fl^{\nabla_0}$; in fact, it extends to the entire Schubert variety $\tld{S}(\tld{z}^*_j,(s_w)_j t_{\omega_j}) \subset \widetilde{\Fl}$.
Since $\tld{\cC}_\sigma = \prod_{j=0}^{f-1} \tld{\cC}_j$, it follows that $x_{j,i}$ extends to $\tld{\cC}_\sigma$.
\end{proof}
\section{Comparison with Hecke operators}
\label{section:hecke-operators}
In this section we provide some additional details regarding the comparison between the functions $x_i \in \cO(\cX(\sigma))$ and Hecke operators which was outlined in Section \ref{section:interpretation}.
The style of this section is informal and we do not provide any proofs. We are aware that a subsequent work \cite{hecke} has realized this section in a detailed and clear way. We recommend the interested readers to that paper. So here, we largely keep this section as when it was first written.
For simplicity, we assume that $K=\bQ_p$, and we keep the same genericity assumptions on $\sigma$ as before.

We have the following diagram.
\[\begin{tikzcd}
	{\mathcal{X}^{\eta,\tau,\text{rig}}_E(t_{w^{-1}(\eta)})} & {\mathcal{X}^{\eta,\tau}(t_{w^{-1}(\eta)})} & {\mathcal{X}^{\eta,\tau}_{\text{red},\mathbb{F}}(t_{w^{-1}(\eta)})} & {\mathcal{X}(\sigma)^{\text{ord}}} \\
	{\mathcal{X}^{\eta,\tau,\text{rig}}_E} & {\mathcal{X}^{\eta,\tau}} & {\mathcal{X}^{\eta,\tau}_{\text{red},\mathbb{F}}} & {\mathcal{X}(\sigma)} \\
	{\operatorname{WD}_\tau^{\text{an}}}
	\arrow[hook', from=2-3, to=2-2]
	\arrow[hook, from=2-1, to=2-2]
	\arrow[hook', from=2-4, to=2-3]
	\arrow[from=2-1, to=3-1]
	\arrow[hook, from=1-2, to=2-2]
	\arrow[hook, from=1-3, to=2-3]
	\arrow[hook, from=1-4, to=2-4]
	\arrow[hook', from=1-3, to=1-2]
	\arrow[hook', from=1-4, to=1-3]
	\arrow["\lrcorner"{anchor=center, pos=0.125, rotate=-90}, draw=none, from=1-3, to=2-2]
	\arrow["\lrcorner"{anchor=center, pos=0.125, rotate=-90}, draw=none, from=1-4, to=2-3]
	\arrow[hook, from=1-1, to=2-1]
	\arrow[hook, from=1-1, to=1-2]
	\arrow["\lrcorner"{anchor=center, pos=0.125}, draw=none, from=1-1, to=2-2]
\end{tikzcd}\]
The reader may find it instructive to compare this diagram to \cite[Diagram (1.3.4)]{fontaine-laffaille} and the discussion therein.
Let us explain some of the symbols appearing in the diagram:
\begin{itemize}
    \item $\eta = (n-1,n-2,\dots,0)$
    \item $\tau : I_{\bQ_p}^{\text{tame}} \to \GL_n(E)$ is the unique principal series type having $\sigma$ as an obvious Serre weight \cite[Section 2]{localmodels}\cite[Section 9.2]{gee-herzig-savitt}.
     We can write $\sigma = F(\mu)$ for some $p$-restricted character $\mu \in X_1(T)$, in which case $\tau = \omega^\mu$, where \(\omega(\gamma) = \gamma((-p)^{1/(p-1)})/(-p)^{1/(p-1)}\) for any $\gamma \in I_{\bQ_p}^{\text{tame}}$.
    \item $\cX^{\eta,\tau}$ is the Emerton-Gee stack for potentially crystalline representations of type $(\eta,\tau)$.
    \item $\cX^{\eta,\tau}(t_{w^{-1}(\eta)}) \subset \cX^{\eta,\tau}$ is the substack of shape $t_{w^{-1}(\eta)}$, where $w$ is an element in the Weyl group in the lowest alcove representation of the Serre weight $\sigma$, as in the first paragraph of \autoref{section:presentation-of-the-irreducible-components}.
    
The shape is an invariant of Breuil-Kisin modules with descent data $\tau$ over $\F$ \cite[Definition 5.1.9]{localmodels}.
Namely, if $\fM$ is a Breuil-Kisin module with descent data, then an \textit{eigenbasis} $\beta$ is a basis for $\fM$ which diagonalizes the descent data.
In this situation we say that $\fM$ has \textit{shape} $t_{\lambda}$ with $\lambda \in X^*(T)$ if the matrix of Frobenius (after undoing the descent data) with respect to the eigenbasis $\beta$ lies in $I v^{\lambda} I$.
This turns out not to depend on the choice of $\beta$.
There is a closely related lift to characteristic zero of this notion \cite[Definition 5.2.4]{localmodels}.\footnote{Note that loc. cit. really works with gauge bases which is closely related to the shape but not equivalent in general. However this turns out to be equivalent in this (extremal) case.}

In \cite[Section 7.2]{localmodels} it is explained how $\cX^{\eta,\tau}$ relates to a moduli stack of Breuil-Kisin modules with descent data of type $(\eta,\tau)$.
Under this map, the substack $\cX^{\eta,\tau}(t_{w^{-1}\eta}) \subset \cX^{\eta,\tau}$ corresponds to the substack of Breuil-Kisin modules of shape $t_{w\eta}$ and satisfying a monodromy condition.
    \item $\operatorname{WD}_\tau$ denotes a moduli stack of Weil-Deligne representations $\rho : W_{\bQ_p} \to \GL_n(E)$ with inertial type $\rho |_{I_{\bQ_p}}\cong\tau$.
    \item $\cX(\sigma)^{\text{ord}} \subset \cX(\sigma)$ is a dense open substack which corresponds to our $\tld{U}_{s_w}$.
    \item Superscript ``rig'' means rigid generic fiber and ``an" means analytification.
\end{itemize}

Via the inertial Langlands correspondence, $\tau \leftrightarrow \sigma(\tau) = \Ind_{B(\F_p)}^{\GL_n(\F_p)} \mu$ (inflated to an $E$-representation of $\GL_n(\bZ_p)$). Given any $i=1,2, \cdots, n$, let $T_i\in \cH(\sigma(\tau))\defeq \operatorname{End}_{\GL_n(\bQ_p)}\left(\text{c-}\Ind_{\GL_n(\bZ_p)}^{\GL_n(\bQ_p)}(\sigma(\tau))\right)$ be the standard Hecke operator supported on the double coset $[\GL_n(\bZ_p) \omega_i(p) \GL_n(\bZ_p)]$, where $\omega_i$ is the $i$-th fundamental cocharacter.
This operator can be identified with certain functions on $\operatorname{WD}_\tau^{\text{an}}$.
It is part of \cite[Theorem 4.1]{6author} that given a Galois representation \(\bar \rho : G_{\bQ_p} \to \GL_n(\F)\), there is an algebra homomorphism
$$\cH(\sigma(\tau))\to R_{\bar \rho}^{\eta, \tau}[1/p] . $$
The maximal ideal spectrum $\mathrm{MaxSpec }R_{\bar \rho}^{\eta, \tau}[1/p]$ parametrizes potentially crystalline lifts of $\bar \rho$ of type $(\eta, \tau)$. 
So given any such $\rho$, we can obtain a map $R_{\bar \rho}^{\eta, \tau}[1/p]\to \overline{\mathbb Q_p}$.
The composite of this map with $\cH(\sigma(\tau))\to R_{\bar \rho}^{\eta, \tau}[1/p]$ gives a character of the Hecke algebra $\cH(\sigma(\tau))$, which is exactly how it acts on the 1-dimensional $\mathrm{Hom}_{\GL_n(\mathbb Z_p)}(\sigma(\tau), \pi(\mathrm{WD}(\rho)))$, where $\pi(\cdot)$ is the local Langlands correspondence for $\GL_n$, defined in \cite[Section 1.8]{6author} .
By computations similar to those of \cite[Section 9-10]{fontaine-laffaille}, this action is given in terms of the Frobenius eigenvalues associated to the Weil-Deligne representation $\operatorname{WD}(\rho)$, normalized by powers of $p$.
This is precisely the value of $T_i$ on the point $\operatorname{WD}(\rho) \in \operatorname{WD}_\tau$.

According to \cite[Equation (5.2.16)]{emerton-gee-hellmann}, the functions $T_i$ on $\operatorname{WD}_\tau^{\text{an}}$ pull back to functions on $\cX^{\eta,\tau,\text{rig}}$ via the above diagram.
In fact the resulting renormalized (i.e. divided by the same powers of $p$) function given by $T_i$ extends to $\cX^{\eta,\tau}(t_{w^{-1}(\eta)})$.
This can be done explicitly: there is a universal Breuil-Kisin module of shape $t_{w^{-1}(\eta)}$ with Frobenius which is given by a triangular matrix, and taking the upper left $(i \times i)$-minor gives a regular function on $\cX^{\eta,\tau}(t_{w^{-1}(\eta)})$.
By triangularity, this function is given by Frobenius eigenvalues, and it restricts to $x_i$ over $\cX(\sigma)^\text{ord}$.
\begin{appendices}
\section{Demazure resolutions}
\label{demazure-resolutions}
We will need to use a modified version of the Demazure resolution as described for example in \cite{faltings-loop-groups} or \cite[Section 8]{pappas-rapoport}. For the appendix we also will shift from the conventions we used so far, as this is the version we will need for our applications. Instead of using left quotients, we will use right quotients, i.e. $\Fl=LG/\bar I$. 
Furthermore, in contrast to the references, we will use the opposite of the standard Iwahori $\bar I$.

In this appendix, we will use \(\widetilde{z} \in W_{a}\) to denote arbitrary elements of the affine Weyl group.
The reader who is only interested in the applications in Section \ref{section:step3} may take \(\widetilde{z}\) as in Definition \ref{*}, for a fixed lowest alcove presentation of a fixed Serre weight. 

\subsection{Demazure resolutions}
If $s \in W_a$ is a simple reflection, then $\bar P_s \subset L \GL_n$ is the subgroup with underlying points $\bar I \cup \bar Is \bar I$.
\begin{defn}[Demazure resolution]
Let \(\tld{z}\in W_a\)
and suppose \(\tld{z} = s_1 s_2\cdots s_k\) is a reduced expression for \(\tld{z}\). The corresponding \textit{Demazure resolution} of \(\tld{z}\) is
\[ \widetilde{D}(s_1,s_2,\dots,s_k) := \bar P_{s_1} \times_{\bar I} \bar P_{s_2} \times_{\bar I} \cdots \times_{\bar I} \bar P_{s_k} ,\]
which by definition is the quotient \((\bar P_{s_1} \times \bar P_{s_2} \times \cdots \times \bar P_{s_k})/ ( \bar I^{k-1}\times \bar I_1)\), where \(\bar I^{k-1}\times \bar I_1\) acts on the right via
 \[ (p_1,p_2,\dots,p_k) \cdot (X_1,X_2, \dots, X_k)  = ( p_1 X_1,X_1^{-1} p_2 X_2,\dots,X_{k-1}^{-1} p_k X_k) . \]
We will sometimes write \(\widetilde{D}(\tld{z}) = \widetilde{D}(s_1,\dots,s_k)\) when there is no danger of confusion, but emphasize that the Demazure resolution depends on a choice of reduced expression.
\end{defn}

\begin{rmk}
The Demazure resolution \(\widetilde{D}(\tld{z})\) we define here is closely related to the Demazure variety \(D(\tld{z}) =  ( P_{s_1} \times P_{s_2} \times \cdots \times P_{s_k})/ \bar I^k \) considered in \cite{pappas-rapoport} and other places. Namely, the only difference is that the rightmost $\bar I$ has been replaced by an $\bar I_1$.
\end{rmk}

There are natural maps \(\pi_{\widetilde{z}} : \widetilde{D}(\widetilde{z}) \to \widetilde{S}(\widetilde{z})\) and \(\pi_{\widetilde{z}}' : D(\widetilde{z}) \to S(\widetilde{z})\) given by multiplying the factors. 
\begin{lemma}
The following square is Cartesian.
\[\begin{tikzcd}
	{\widetilde{D}(\tld{z})} & {D(\tld{z})} \\
	{\widetilde{S}(\tld{z})} & {S(\tld{z})}
	\arrow[from=1-1, to=2-1, "\pi_{\tld{z}}"]
	\arrow[from=2-1, to=2-2,"q"]
	\arrow[from=1-2, to=2-2,"\pi_{\tld{z}}'"]
	\arrow[from=1-1, to=1-2,"q'"]
	\arrow["\lrcorner"{anchor=center, pos=0.125}, draw=none, from=1-1, to=2-2]
\end{tikzcd}\]
\end{lemma}
\begin{proof}
Let $D(\tld{z})\times_{S(\tld{z})} {\widetilde{S}(\tld{z})}$ be the fiber product. This is a $T$-torsor over ${\widetilde{S}(\tld{z})}$ and there is a morphism $\widetilde{D}(\tld{z})\to D(\tld{z})\times_{S(\tld{z})} {\widetilde{S}(\tld{z})} $  of $T$-torsors over ${\widetilde{S}(\tld{z})}$. Hence it is an isomorphism.
\end{proof}
This essentially allows us to reduce relevant geometric questions about $\widetilde{D}(\tld{z})$ to ones about $D(\tld{z})$. For example, recall that the closed Schubert cell \(\widetilde{S}(\tld{z}) \) associated to \(\tld{z}\) is the closure of the open Schubert cell \( \widetilde{S}^\circ(\tld{z}) =\bar I \tld{z} \bar I/ \bar I_1\) inside the flag variety \( \widetilde{\Fl} \).

\begin{rmk}\label{isomorphiccells}
Inside $\widetilde{D}(s_1,s_2,\dots,s_k)$ we have the open cell 
$$ \widetilde{D}^\circ(s_1,s_2,\dots,s_k) := \bar I s_1 \bar I \times_{\bar I} \bar I s_2 \bar I \times_{\bar I} \cdots \times_{\bar I} \bar I s_k \bar I_1 $$
which maps isomorphically to $\widetilde{S}^\circ(\tld{z})=\bar I \tld{z} \bar I_1$ via $\pi_{\tld{z}}$.
\end{rmk}

\begin{prop}[Compare {\cite[Proposition 8.8]{pappas-rapoport}}] \label{main0}
The map \(\pi_{\tld{z}} : \widetilde{D}(\widetilde{z}) \to \widetilde{S}(\widetilde{z})\) is surjective, proper, and \(\mathcal{O}_{\widetilde{S}({\tld{z}})} \lrisom (\pi_{\tld{z}})_* \mathcal{O}_{\widetilde{D}({\tld{z}})}\) is an isomorphism.
Moreover, \(\widetilde{D}({\tld{z}})\) is a smooth projective variety over \(\mathbb{F}\) of dimension \( l({\tld{z}}) + \operatorname{dim}(T) \), where \(l({\tld{z}})\) is the length of \({\tld{z}}\). 
\end{prop}
\begin{proof}
Only the statement about functions is not clear from \cite{pappas-rapoport} and the above lemma.
For this we use flat base change \cite[\href{https://stacks.math.columbia.edu/tag/02KH}{Tag 02KH}]{stacks-project}.
By flatness of \(q\) there is an isomorphism
\begin{equation}
	\label{eq:flat-base-change}
	(q')^* R(\pi_{\tld{z}}')_* \lrisom R(\pi_{\tld{z}})_* q^* .
\end{equation}
By \cite[Proposition 9.7]{pappas-rapoport}, \(R^i(\pi_{\tld{z}}')_* \mathcal{O}_{D({\tld{z}})} = 0\) for \(i > 0\) and \(\mathcal{O}_{S({\tld{z}})} \cong (\pi_{\tld{z}}')_* \mathcal{O}_{D({\tld{z}})}\).
Applying \eqref{eq:flat-base-change} to \(\mathcal{O}_{D({\tld{z}})}\), we get an isomorphism
\[\mathcal{O}_{\tld{S}({\tld{z}})} \cong (q')^* \mathcal{O}_{S({\tld{z}})} \cong (q')^* (\pi_{\tld{z}}')_* \mathcal{O}_{D({\tld{z}})} \lrisom R(\pi_{\tld{z}})_* q^* \mathcal{O}_{D({\tld{z}})} \cong R(\pi_{\tld{z}})_* \mathcal{O}_{\tld{D}({\tld{z}})}.\]
From this we see that in fact \(R^i(\pi_{\tld{z}})_* \mathcal{O}_{\tld{D}({\tld{z}})} = 0\) for \(i > 0\) and we have an isomorphism \( \mathcal{O}_{\tld{S}({\tld{z}})} \lrisom (\pi_{\tld{z}})_* \mathcal{O}_{\tld{D}({\tld{z}})}\).
Tracing through the definitions one can check that this map is the obvious one.
\end{proof}

\begin{prop}[Compare {\cite[Proposition 9.6]{pappas-rapoport}}]
Let \({\tld{z}},{\tld{z}}' \in W_a\) and assume that \({\tld{z}}' \leq {\tld{z}} \) in the Bruhat order.
Choose reduced expressions \({\tld{z}} = s_1 s_2 \cdots s_k\) and \({\tld{z}}' = s_{i_1} s_{i_2} \cdots s_{i_l}\), where \(1 \leq i_1 < i_2 < \cdots < i_l \leq k\).
Then the map 
\[\begin{aligned}
\sigma_{{\tld{z}}',{\tld{z}}} : \widetilde{D}({\tld{z}}')=\widetilde{D}(s_{i_1},\dots,s_{i_l}) & \to \widetilde{D}(s_1,s_2,\dots,s_k)=\widetilde{D}({\tld{z}})
\end{aligned}\]
that inserts 1's in coordinates not corresponding to any \(i_j\), satisfies the following properties:
\begin{enumerate}
\item \(\sigma_{{\tld{z}}',{\tld{z}}}\) is a closed immersion.
\item The diagram
\[\begin{tikzcd}
\widetilde{D}({\tld{z}}') \arrow[r,"\sigma_{{\tld{z}}',{\tld{z}}}"] \arrow[d,"\pi_{{\tld{z}}'}"]  & \widetilde{D}({\tld{z}}) \arrow[d,"\pi_{\tld{z}}"] \\
\widetilde{S}({\tld{z}}') \arrow[r,hookrightarrow] & \widetilde{S}({\tld{z}})
\end{tikzcd}\]
is commutative.
\item If \({\tld{z}}' \leq {\tld{z}}'' \leq {\tld{z}}\) and \({\tld{z}}''\) is given a reduced expression in between those chosen for \({\tld{z}}'\) and \({\tld{z}}\), then \(\sigma_{{\tld{z}}',{\tld{z}}} = \sigma_{{\tld{z}}'',{\tld{z}}} \circ \sigma_{{\tld{z}}'',{\tld{z}}'}\).
\end{enumerate}
\end{prop}

\subsection{Intermediate Demazure resolutions} \label{Intermediate Demazure resolutions}
Assume that ${\tld{z}}$ has length $k$ with a reduced expression ${\tld{z}}=s_1s_2\cdots s_k$. Then we write ${\tld{z}}_i=s_1\cdots s_{i-1}s_{i+1}\cdots s_k$. Every ${\tld{z}}'\leq \tld{z}$ of length $k-1$ has a reduced expression  of the form ${\tld{z}}'=s_1s_2\cdots s_{i-1}s_{i+1} \cdots s_k$, i.e. is equal to at least one ${\tld{z}}_i$.

\begin{defn}
\begin{itemize}
    \item We define $\widetilde{D}^\circ({\tld{z}},i)$ as the following subscheme of $\widetilde{D}({\tld{z}})$ :
$$ \bar I s_1 \bar I \times_{\bar I} \bar I s_2 \bar I \times_{\bar I} \cdots \times_{\bar I} \bar P_{s_i} \times_{\bar I} \cdots \times_{\bar I} \bar I s_k \bar I_1 \subset \bar P_{s_1} \times_{\bar I} \bar P_{s_2} \times_{\bar I} \cdots \times_{\bar I} \bar P_{s_k} =\widetilde{D}({\tld{z}}). $$
\item We define $\widetilde{D}^\circ(i)$ as the following subscheme of $\widetilde{D}({\tld{z}})$ :
$$ \bar I s_1 \bar I \times_{\bar I} \bar I s_2 \bar I \times_{\bar I} \cdots \times_{\bar I} \bar I \times_{\bar I} \cdots \times_{\bar I} \bar I s_k \bar I_1 \subset \bar P_{s_1} \times_{\bar I} \bar P_{s_2} \times_{\bar I} \cdots \times_{\bar I} \bar P_{s_k} =\widetilde{D}({\tld{z}}). $$

\end{itemize}
\end{defn}
\begin{rmk}
    Note that $\widetilde{D}^\circ({\tld{z}},i)= \widetilde{D}^\circ({\tld{z}})\cup \widetilde{D}^\circ(i)$.
\end{rmk}

\begin{lemma}\label{pullback}
Let ${\tld{z}}$ be as above. For an element ${\tld{z}}'\leq {\tld{z}}$ of length $l({\tld{z}})-1$, we denote $$\cJ({\tld{z}}')=\{ i\in \{1,...,k\} | {\tld{z}}'={\tld{z}}_i \}.$$
Under the inclusion $\widetilde{D}^\circ({\tld{z}},i)\xrightarrow[]{} \widetilde{D}({\tld{z}})$, the following diagram is cartesian
\[\begin{tikzcd}
	{\bigcup_{i\in \cJ({\tld{z}}')} \widetilde{D}^\circ({\tld{z}},i)} & {\widetilde{D}({\tld{z}})} \\
	{\widetilde{S}^\circ({\tld{z}})\cup\widetilde{S}^\circ({\tld{z}}') } & {\widetilde{S}({\tld{z}}).}
	\arrow[from=1-1, to=2-1]
	\arrow[hook, from=1-1, to=1-2]
	\arrow[hook, from=2-1, to=2-2]
	\arrow[from=1-2, to=2-2]
\end{tikzcd}\]
\end{lemma}

\begin{proof}
We will show that 
$\pi_{\tld{z}}(\widetilde{D}^\circ({\tld{z}},i))=
\widetilde{S}^\circ({\tld{z}})\cup\widetilde{S}^\circ({\tld{z}} _i)$. The above lemma then follows.

To verify the equality, we simply note that
$$ \bar I s_1 \bar I\cdots \bar I s_m\bar I=\bar I{\tld{z}} \bar I$$
for any ${\tld{z}}\in W_a$ and a reduced expression $s_1\cdots s_m$ of it.
\end{proof}

Let $s\in S$ be a simple reflection with its parahoric $\bar P_s$.
 In what follows we will prove that the first two columns of table \ref{analogies} are isomorphic. Later we will extend this analogy to exhibit "charts" of $\widetilde{D}^\circ({\tld{z}},i)$.

\begin{table}[h!]
\bgroup
\def\arraystretch{1.5}
\begin{tabular}{|l|l|l|}
\hline
pt        & $\bar I/ \bar I$    &   $\widetilde{D}^\circ(i) $    \\ \hline
$\bA^1_0$ & $\bar I s \bar I/ \bar I$  & $\operatorname{im}(\phi_{\tld{z}})$    \\ \hline
$\bA^1_\infty$   & $s \bar I s \bar I/ \bar I$ & $\operatorname{im}(\psi_{\tld{z}})$ \\ \hline
$\bP^1$   & $\bar P_s / \bar I$  & $\widetilde{D}^\circ({\tld{z}},i)$ \\ \hline
\end{tabular}   
 \caption{Analogies of the geometric objects}
  \label{analogies}
\egroup
\end{table}

Recall that the affine Weyl group $W_a$ is generated, as Coxeter group, by the reflections along the walls of any alcove.
For our purposes it is natural to use the dominant base alcove $A_0$, because $\bar I$ is the stabilizer of this alcove. 
Thus our set $S$ of simple reflections generating $W_a$ are the spherical reflections $s_\alpha$ for $\alpha_{i(i+1)} = \cE_i - \cE_{i+1}$, $1 \leq i \leq n-1$ a simple root, and the affine reflection $s_a = t_{\alpha_{1n}} s_{\alpha_{1n}}$, where $\alpha_{1n} = \cE_1 - \cE_n$ is the highest root.
Corresponding to these simple reflections are certain affine root groups:
\begin{defn} \label{gooddef}
Let $s\in S$.
\begin{itemize}
    \item We define the following roots associated to $s$:
    $$ \bar \alpha_s \defeq \begin{cases} - \alpha_{i(i+1)} = -\cE_i + \cE_{i+1}, & \text{ if }s = s_{\alpha_{i(i+1)}} \\
\alpha_{1n}=  \cE_1 - \cE_n , & \text{ if } s = s_a . \end{cases} $$
    \item We define the affine root group $\bU_s$ as follows.
    For $s$ corresponding to a spherical (resp. affine) simple reflection, we  let $\bU_s \subset \bar I$ be the subscheme defined by valuation bounds $[0,0]$ (resp. $[1,1]$)  in position corresponding to $\bar \alpha_s$ and $\emptyset$ in all other entries.\footnote{The meaning of being defined by valuation bounds is given in Definition \ref{defined-by-valuation-bounds}.}
\end{itemize}
\end{defn}
As for ordinary root groups, we have for each $s \in S$ a canonical homomorphism $u_s : \bG_a \lrisom \bU_s$ such that
$$ t u_s(a) t^{-1} = u_s ( \bar \alpha_s (t) a ), \,\,\,\,\, \text{ for any }t \in T. $$
The definition of $\bU_s$ is made so that the following proposition holds.
\begin{prop}
\label{affine-root-group-lemma}
For $s \in S$ one of the simple reflections described above, we have isomorphisms
\begin{align*}
      \bG_a  \xrightarrow[]{u_s} \bU_s s &  \xrightarrow[]{\phi_s}  \bar Is \bar I/ \bar I \\
    u\cdot s &\mapsto  u\cdot s \bar I
\end{align*}
which can be extended to an isomorphism 
\begin{align*}
    \widetilde{\phi_s}: T \times \bU_s s&\xrightarrow[]{\cong}  \bar Is \bar I / \bar I_1\\
    (t, u\cdot s) &\mapsto u \cdot s\cdot t \bar I_1
\end{align*}
\end{prop}
\begin{proof}
Note that Proposition \ref{iwahori-valuation-bounds} is directly applicable.
By the orbit-stabilizer theorem, we have an isomorphism
$$ \ovl{I} s \ovl{I} / \ovl{I} \cong \ovl{I} / (s\ovl{I}s^{-1} \cap \ovl{I}) . $$
By Proposition \ref{iwahori-valuation-bounds}, both $\ovl{I}$ and $s \ovl{I} s^{-1} \cap \ovl{I}$ are defined by valuation bounds which only differ in one place:
\begin{itemize}
\item In the case that $s = s_{\alpha_{i(i+1)}}$ is a spherical reflection, $s \ovl{I} s^{-1} \cap \ovl{I}$ has valuation bound $[1,\infty]$ in position $-\alpha_{i(i+1)}$ whereas $\ovl{I}$ has valuation bound $[0,\infty]$ in this position.
\item In the case that $s = t_{\alpha_{1n}} s_{\alpha_{1n}}$ is the affine reflection, $s \ovl{I} s^{-1} \cap \ovl{I}$ has valuation bound $[2,\infty]$ in position $\alpha_{1n}$ whereas $\ovl{I}$ has valuation bound $[1,\infty]$ in this position.
\end{itemize}
The result then follows from Iwahori decomposition.
\end{proof}
\begin{rmk}
The proof actually shows that $ \bU_s(R)s  \lrisom \bar I(R) s \bar I(R) / \bar I(R)$ is an isomorphism for each $\F$-algebra $R$.
In particular, $\bar I s \bar I / \bar I$ is just the presheaf quotient.
\end{rmk}
Having shown that $\bar Is \bar I/ \bar I$ is isomorphic to $\bA^1$, and therefore also its translate $ s \bar Is \bar I / \bar I \subseteq \bar P_s / \bar I$, let us show that they glue together correctly to give an isomorphism $\bar P_s / \bar I \cong \bP^1$.
To prove this we need to specify precisely the chosen representatives for $s \in S$ inside $L\GL_n$.
We make the following choices:
$$ s_{\alpha_{i(i+1)}} = \begin{pmatrix}
1 & & & & & \\
& \ddots & & & & \\
& & 0 & - 1 & & \\
& & 1 & 0 & & \\
& & & & \ddots & \\
& & & & & 1
\end{pmatrix}, \,\,\,\,\,
s_a = \begin{pmatrix}
0 & & & & v \\
& 1 & & & \\
& & \ddots & & \\
& & & 1 & \\
- v^{-1} & & & & 0
\end{pmatrix} .
$$

\begin{lemma}
\label{P1-lemma}
\begin{enumerate}
    \item For $a\in \bG_m(R)$ and a simple reflection $s$, we have the following equality
    $$u_s(a)\cdot s\cdot  u_s(a^{-1})\cdot  s\cdot  u_s(a)\cdot  s=a^{\bar \alpha_s^\vee}. $$
    \item The two isomorphisms
    \begin{align*}
\delta_0:\bA^1_0 & \lrisom \bar I s \bar I / \bar I & \delta_\infty: \bA^1_\infty & \lrisom s \bar I s \bar I  / \bar I \\
a & \mapsto  u_s(a)  s \bar I& b & \mapsto s u_s (-b) s \bar I .
\end{align*}
glue together to an isomorphism $\bP^1 \cong \bar P_s/ \bar I$.
\end{enumerate}
\end{lemma}
\begin{proof}
\begin{enumerate}
    \item This can be reduced to an $SL_2$ computation. First, assume $s$ is spherical. Then we can restrict to $\langle e_i,e_{i+1}\rangle$ and the computation becomes
    $$\begin{pmatrix} 1 & 0 \\ a & 1 \end{pmatrix} \begin{pmatrix} 0 & -1 \\ 1 & 0 \end{pmatrix} \begin{pmatrix} 1 & 0 \\ a^{-1} & 1 \end{pmatrix} \begin{pmatrix} 0 & -1 \\ 1 & 0 \end{pmatrix} \begin{pmatrix} 1 & 0 \\ a & 1 \end{pmatrix} \begin{pmatrix} 0 & -1 \\ 1 & 0 \end{pmatrix}= \begin{pmatrix} a^{-1} & 0 \\ 0 & a \end{pmatrix}. $$
    Now assume $s$ is the affine reflection. Then we can restrict to $\langle e_1,e_{n}\rangle$ and the computation becomes 
        $$\begin{pmatrix} 1 & av \\ 0 & 1 \end{pmatrix} \begin{pmatrix} 0 & v \\ -v^{-1} & 0 \end{pmatrix} \begin{pmatrix} 1 & \frac{v}{a} \\ 0 & 1 \end{pmatrix} \begin{pmatrix} 0 & v \\ -v^{-1} & 0 \end{pmatrix} \begin{pmatrix} 1 & av \\ 0 & 1 \end{pmatrix} \begin{pmatrix} 0 & v \\ -v^{-1} & 0 \end{pmatrix}= \begin{pmatrix} a & 0 \\ 0 & a^{-1} \end{pmatrix}. $$
    \item We can rewrite the equality from the first part as
    \begin{equation*}\label{transfunc}
        \delta_\infty(a)= (s u_s (-a) s) \bar I = (u_s(a^{-1})  s X) \bar I=\delta_0(a^{-1})
    \end{equation*}
    where $X = u_s(a) (-a)^{\bar \alpha_s^\vee} \in \bar I$. In fact
    $$ \im(\delta_0)\cap \im(\delta_\infty)=\bar I s \bar I-\{s\}=s\bar I s \bar I-\{1\}$$
    and the maps glue together to the desired isomorphism.
\end{enumerate}
\end{proof}
Now let us go from the case of $ \bar P_s/ \bar I$ to the intermediate Demazure resolution $\widetilde{D}^\circ({\tld{z}},i)$. This is the object that we want to cover by two affine opens.

To do so, let us define two maps: 

\begin{defn}
Define 
$$\widetilde{\phi_{\tld{z}}}: \bU_{s_1}s_1 \times \bU_{s_2}s_2  \times \cdots \times  \bU_{s_k}s_k\times T \to \widetilde{D}^\circ({\tld{z}},i) ,$$
$$\widetilde{\psi_{\tld{z}}}: \bU_{s_1}s_1 \times  \bU_{s_2}s_2 \times \cdots \times s_i \bU_{s_i} s_i  \times \cdots \times  \bU_{s_k}s_k \times T \to \widetilde{D}^\circ({\tld{z}},i)$$
as composition
$$\bU_{s_1}s_1 \times \bU_{s_2}s_2  \times \cdots \times  \bU_{s_k}s_k\times T \to  \bar I s_1 \bar I\times \cdots\times \bar P_{s_i}\times \cdots \times \bar I s_k \bar I_1                 \to \widetilde{D}^\circ({\tld{z}},i)$$
where the first map comes from Lemma \ref{affine-root-group-lemma} and the second map is the quotient map. $\widetilde{\psi_{\tld{z}}}$ is  then defined in a similar fashion but embedding $s_i \bU_{s_i} s_i$ instead of $ \bU_{s_i} s_i$ into $\bar P_{s_i}$.
\end{defn}

\begin{rmk}
\begin{itemize}
By omitting the torus we also get maps
$$ \phi_{\tld{z}}: \bU_{s_1}s_1 \times \bU_{s_2}s_2  \times \cdots \times  \bU_{s_k}s_k \to {D}^\circ({\tld{z}},i)$$
and similarly for $\psi_{\tld{z}}$ such that $\phi_{\tld{z}}\times \id_T=\widetilde{\phi_{\tld{z}}}$ and $\psi_{\tld{z}}\times \id_T=\widetilde{\psi_{\tld{z}}}$.
\end{itemize}
\end{rmk}

\begin{lemma}
\label{open-immersion-1}
The above map $\widetilde{\phi_{\tld{z}}}$ is an open immersion with image equal to $\widetilde{D}^\circ({\tld{z}})$.
\end{lemma}
\begin{proof}
It suffices to show this for $\phi_{\tld{z}}$. We will show that $\phi_{\tld{z}}$ is a monomorphism regarded as natural transformation between the functors representing the above varieties (say on $\mathbb{F}$-algebras). Since sheafification is left-exact, it preserves monomorphisms. Therefore it suffices to show that $\phi_{\tld{z}}$ is a monomorphism before sheafifying the target, i.e. we can just take the naive quotient by $ \bar I^{k-1}\times \bar I_1(R)$. Now we compute in explicit coordinates. Suppose $\phi_{\tld{z}}( u_1s_1, \cdots u_k s_k)=\phi_{\tld{z}}( u'_1s_1,\cdots u'_ks_k )$ for $u_l,u'_l \in \bU_{s_l}(R)$. Hence there exist $X_k\in \bar I_1(R)$ and $X_1 \cdots, X_{k-1} \in I(R)$, such that $$(u_1 s_1,\cdots, u_k s_k)=(u'_1 s_1 X_1 , X_1^{-1} u'_2 s_2 X_2,\cdots, X_{k-1}^{-1} u'_k s_k X_k). $$ 
Thus the equality
$$ u_1 s_1 =u'_1 s_1 X_1 $$
implies an equality $  u_1 s_1 \cdot \bar I=  u_1' s_1\cdot \bar I$ inside $ \bar I s  \bar I/ \bar I$. But by \ref{affine-root-group-lemma}, this already implies $u_1=u_1'$. For $u_2, u'_2$ we get a similar equation and therefore we obtain inductively that $X_l=1, u_l=u'_l$ for all $l$, i.e. $\phi_{\tld{z}}$ is indeed a monomorphism.

That its image is equal to the open cell, follows from
$$ {D}^\circ({\tld{z}})\cong {S}^\circ({\tld{z}})=\bar I {\tld{z}}\bar I_1=\bar I s_1\bar Is_2\bar I \cdots \bar Is_k\bar I.$$
\end{proof}

\begin{lemma}
The map $\widetilde{\psi_{\tld{z}}}$ is an open immersion and its image contains $\widetilde{D}^\circ(i) $.
\end{lemma}

\begin{proof}
The second part is immediate as $1\in P_s$ is contained in the image of $\psi_{\tld{z}}$. For the rest of the proof we note that since the source of $\psi_{\tld{z}}$ is locally of finite type over $\mathbb{F}$, so will be the morphism $\psi_{\tld{z}}$ itself (see \cite[\href{https://stacks.math.columbia.edu/tag/01T8}{Tag 01T8}]{stacks-project}) and we will implicitly use this throughout the proof. To prove the first part, let us first show that it is an monomorphism. 

Now we proceed as in the proof from the previous lemma. Therefore let us suppose that  $\psi_{\tld{z}}( u_1s_1, \cdots, s_i u_i s_i, \cdots, u_k s_k)=\psi_{\tld{z}}( u'_1s_1,\cdots, s_i u'_i s_i, \cdots, u'_ks_k )$ for $u_l,u'_l \in \bU_{s_l}(R)$.  
By the definition of $D({\tld{z}}) $, there exist $X_i\in I(R)$, such that 
$$(u_1 s_1,\cdots, s_i u_i s_i, \cdots ,s_ku_k)=(u'_1 s_1 X_1 , X_1^{-1} u'_2 s_2 X_2,\cdots, X_{i-1}^{-1} s_i u'_i s_i X_i, \cdots, X_{k-1}^{-1} u'_k s_k X_k). $$ 
Then we can obtain inductively that $u_1'=u_1, \cdots, u'_{i-1}=u_{i-1}$ and $ X_1=\cdots=X_{i-1}=1.$ For $u_i,u'_i$ we can argue in a similar way since \ref{affine-root-group-lemma} implies that $s\bU s\cong s\bar I s \bar I/\bar I$. And we proceed again inductively to show that $X_l=1, u_l=u'_l$ for all $l$, i.e. $\psi_{\tld{z}}$ is indeed a monomorphism.

Having proven that it is indeed a monomorphism, by \cite[\href{https://stacks.math.columbia.edu/tag/05VH}{Tag 05VH}]{stacks-project} we conclude that $\psi_{\tld{z}}$ is in fact unramified. By \cite[\href{https://stacks.math.columbia.edu/tag/039H}{Tag 039H}]{stacks-project} this implies that it is a surjection on the completion of the local rings. Since both source and target of the morphism are smooth of the same dimension, this implies that on the completion of local rings, $\psi_{\tld{z}}$ is an isomorphism. In particular it is \'etale (see also \cite[\href{https://stacks.math.columbia.edu/tag/039M}{Tag 039M}]{stacks-project}). Therefore $\psi_{\tld{z}}$ is an \'etale monomorphism and thus an open immersion \cite[\href{https://stacks.math.columbia.edu/tag/025G}{Tag 025G}]{stacks-project}. 
\end{proof} 

\begin{rmk}
The reason that we insert $s_i \bU_{s_i} s_i$ at the $i$-th position of the domain of $\widetilde{\psi_{\tld{z}}}$ rather than just delete the $i$-th position is that we want $\widetilde{\psi_{\tld{z}}}$ to be an open immersion.
\end{rmk}

\subsection{Extending functions} \label{Extending functions}

Ultimately, we are interested in showing that certain functions on $\tld{S}^\circ({\tld{z}})$ extend to $\tld{S}({\tld{z}})$. Let us first show that this problem can be reduced to extending functions on intermediate Demazure resolutions.
\begin{lemma}\label{extend}
    Let $f:\widetilde{S}^\circ({\tld{z}})\to \bA^1 $ be a function and denote by $f':\widetilde{D}^\circ({\tld{z}})\to \bA^1$ its pullback along the isomorphism $\widetilde{D}^\circ(\tld{{\tld{z}}})\cong \widetilde{S}^\circ({\tld{z}})$. 
    Assume that for all $i\in \{1,..,l({\tld{z}})\}$  the function $f'$ extends
\[\begin{tikzcd}
	{\widetilde{D}^\circ({\tld{z}})} & {\widetilde{D}^\circ({\tld{z}},i)} \\
	{\bA^1.}
	\arrow[hook, from=1-1, to=1-2]
	\arrow["{f'}"', from=1-1, to=2-1]
	\arrow[dashed, from=1-2, to=2-1]
\end{tikzcd}\]
Then the function $f$ extends
\[\begin{tikzcd}
	{\widetilde{S}^\circ({\tld{z}})} & {\widetilde{S}({\tld{z}})} \\
	{\bA^1.}
	\arrow[hook, from=1-1, to=1-2]
	\arrow["{f}"', from=1-1, to=2-1]
	\arrow[dashed, from=1-2, to=2-1]
\end{tikzcd}\]
\end{lemma}

\begin{proof}
Since Schubert varieties are normal by \cite[Theorem 0.3]{pappas-rapoport} $\tld{S}(\tld{z}^*)$, we can use the algebraic version of the Hartog's lemma (see \cite[Tag 031T]{stacks-project}) and it suffices to extend along the codimension $1$ strata. So it suffices to extend the function $f$ from $\widetilde{S}^\circ({\tld{z}})$ to
$\widetilde{S}^\circ({\tld{z}}) \cup \bigcup_\cI \widetilde{S}^\circ({\tld{z}}')$, where $\cI=\{{\tld{z}}' | {\tld{z}}'<{\tld{z}} \ , \  l({\tld{z}}')=l({\tld{z}})-1 \}$. Note that the latter space is an open subscheme of $\widetilde{S}({\tld{z}})$ by the stratification property. 
    Now that by Lemma \ref{pullback} the following diagram is cartesian
\[\begin{tikzcd}
	{\widetilde{D}^\circ({\tld{z}})} & {\bigcup_{i\in \{1,...,l({\tld{z}})\}} \widetilde{D}^\circ({\tld{z}},i)} & {\widetilde{D}({\tld{z}})} \\
	{\widetilde{S}^\circ({\tld{z}})} & {\widetilde{S}^\circ({\tld{z}})\cup\bigcup_\cI \widetilde{S}^\circ({\tld{z}}')} & {\widetilde{S}({\tld{z}}).}
	\arrow["\cong"', from=1-1, to=2-1]
	\arrow[hook, from=1-1, to=1-2]
	\arrow[hook, from=2-1, to=2-2]
	\arrow[from=1-2, to=2-2]
	\arrow[hook, from=1-2, to=1-3]
	\arrow[hook, from=2-2, to=2-3]
	\arrow[from=1-3, to=2-3]
	\arrow["\lrcorner"{anchor=center, pos=0.125}, draw=none, from=1-2, to=2-3]
\end{tikzcd}\]
By Proposition \ref{main0}, the rightmost morphism fulfills $\mathcal{O}_{\widetilde{S}({\tld{z}})} \lrisom (\pi_{{\tld{z}}})_* \mathcal{O}_{\widetilde{D}({\tld{z}})}$ and therefore so does its pullback along the bottom right inclusion, in particular we get an isomorphism on global sections.

By our assumption $f'$ extends to a function on $\bigcup_\cI \widetilde{D}^\circ({\tld{z}},{\tld{z}}')$. Since we have the aforementioned isomorphism on global sections, this extends $f$ from $\widetilde{S}^\circ({\tld{z}})$ to
$\widetilde{S}^\circ({\tld{z}}) \cup \bigcup_\cI \widetilde{S}^\circ({\tld{z}}')$ and we are done.
\end{proof}

The functions we consider are all examples of the following construction:
Consider the isomorphisms
\begin{align*}
\tld{D}^\circ({\tld{z}}) \lrisom \tld{S}^\circ({\tld{z}}) & \lrisom  S^\circ({\tld{z}})\times T \\
A {\tld{z}} t \bar I_1 & \mapsto ( A {\tld{z}} \bar I_1, t) \,\,\,\,\, (A \in \bar I_1) .
\end{align*}
Here the first isomorphism $(X_1,\dots,X_k) \mapsto X_1 \cdots X_k \bar I_1$ is due to remark \ref{isomorphiccells} and the second isomorphism corresponds to a trivialization of the $T$-torsor $\tld{S}^\circ({\tld{z}}) \to S^\circ({\tld{z}})$.
\begin{defn}
\label{function-on-demazure-resolution}
For any function $f : T \to \bA^1$, define $x_f : \tld{D}^\circ({\tld{z}}) \to \bA^1$ as the composite map 
$x_f : \tld{D}^\circ({\tld{z}}) \lrisom  S^\circ({\tld{z}}) \times T \stackrel{\pr_T}{\longrightarrow} T \stackrel{f}{\longrightarrow} \bA^1 ,$
where the first isomorphism is the one above.
\end{defn}
\begin{lemma}
\label{function-of-t}
Suppose $t \in T(R)$ and $X_i \in \ovl{I}_1(R) s_i \ovl{I}_1(R)$ for $1 \leq i \leq k$. Then
$$ x_f(X_1,X_2,\dots,X_k t) = f(t) . $$
\end{lemma}
\begin{proof}
Using the fact that $\bar I_1 s_1 \bar I_1 s_2 \bar I_1 \cdots \bar I_1 s_k \bar I_1 \subseteq \bar I_1 {\tld{z}} \bar I_1$ \cite[Corollary 2.5]{propiwa},
we have $X_1X_2 \cdots X_k \subseteq \bar I_1 {\tld{z}} \bar I_1$.
The conclusion is then straightforward from the definitions.
\end{proof}
\begin{rmk}
It is immediate from the lemma that if $\widetilde{\phi_{\tld{z}}} :\bA^k\times T \lrisom \tld{D}^\circ({\tld{z}})$ is the isomorphism from Lemma \ref{open-immersion-1}, then $(\widetilde{\phi_{\tld{z}}})^* x_f(a,t) = f(t)$.
\end{rmk}
We now want to consider the problem of extending $x_f$ from $\tld{D}^\circ({\tld{z}})$ to $\tld{D}({\tld{z}},i)$.
Using the open chart $\widetilde{\psi_{\tld{z}}} :  \bA^k \times T \hookrightarrow \tld{D}({\tld{z}},i)$, which covers $\tld{D}({\tld{z}},i) \setminus \tld{D}^\circ({\tld{z}})$, we see that $x_f$ extends if and only if the rational function $(\widetilde{\psi_{\tld{z}}})^* x_f :  \bA^k\times T \to \bP^1$ is regular.
Therefore we would like to have an explicit formula for $\psi_{\tld{z}}^* x_f$. We first recall the following definition.

\begin{defn}
    We define the \emph{finite} Weyl group action on $\lambda \in X^*(T)$ as follows through generators of $W_a$ (see the paragraph before Definition\autoref{gooddef}). If $s$ is a spherical reflection, then $s\lambda$ is given by the usual reflection action. If $s=t_\alpha s_\alpha$ is an affine reflection with $\alpha$ being the highest root, then $s\lambda$ is defined to be $s_\alpha(\lambda)$.
\end{defn}

\begin{lemma}
\label{function-in-other-chart}
For any $(a,t) \in  \bA^k\times T$ we have 
$$ \widetilde{\psi_{\tld{z}}}^* x_f (a,t) = f\left( a_i^{ s_k^{-1} s_{k-1}^{-1}\dots s_{i+1}^{-1}\bar\alpha_{i}^\vee} t \right) ,$$
where the action of each \(s_{j}^{-1}\) is via the finite Weyl group action. 
\end{lemma}
\begin{proof}
Let $(a,t) \in  \bA^k \times T$ and $(a',t') \in  \bA^k \times T$ and assume that $\widetilde{\psi_{\tld{z}}}(a,t) = \widetilde{\phi_{\tld{z}}}(a',t')$.
Spelling it out, this means that for some $X_k \in \bar I_1$ and $X_1,X_3,\dots,X_{k-1} \in \bar I$ we have
$$ \left( u_1(a_1)s_1 ,\dots, s_i u_i(a_i) s_i,\dots,u_k(a_k)s_k t \right) = \left( u_1(a_1') s_1 X_1,\dots,X_{i-1}^{-1}  u_i(a_i')s_i X_{i},\dots,X_{k-1}^{-1}  u_k(a_k') s_k t' X_{k}\right) . $$
(Here we use the simplifying notation $u_i = u_{s_i}$ and $\bar \alpha_i=\bar \alpha_{s_i}$).
Like in the proof of Lemma \ref{open-immersion-1}, we obtain $X_1,X_{2},\dots,X_{i-1}=1$ and $(a_1,a_{2},\dots,a_{i-1}) = (a_1',a_{2}',\dots,a_{i-1}')$.
Now at $i$ we get $$ s_i u_i(a_i) s_i =  u_i(a_i') s_i X_i. $$
By Lemma \ref{P1-lemma} it follows that $a_i' = - a_i^{-1}$ and $X_i =  u_i(-a_i) a_i^{\bar \alpha_{i}^\vee}$.
Hence
\begin{align*}
\widetilde{\psi_{\tld{z}}}(a,t) & = \left(\dots, s_i u_i(a_i) s_i, u_{i+1}(a_{i+1})s_{i+1},  \dots \right) \\
& = \left(\dots,  u_i(-a_i^{-1})s_i u_i(-a_i) a_i^{\bar \alpha_{i}^\vee},  u_{i+1}(a_{i+1})s_{i+1}, \dots \right) \\
& \simeq \left(\dots,  u_i(-a_i^{-1})s_i u_i(-a_i), a_i^{\bar \alpha_{i}^\vee} u_{i+1}(a_{i+1}) s_{i+1}, \dots \right). 
\end{align*}
The crucial computation is now that
$$ a_i^{\bar \alpha_{i}^\vee}  u_{i+1}(a_{i+1})s_{i+1}=  u_{i+1}(a_i^{\langle \bar \alpha_{i+1}, \bar \alpha_{i}^\vee \rangle}a_{i+1})  a_i^{\bar \alpha_{i}^\vee} s_{i+1}= u_{i+1}(a_i^{\langle \bar \alpha_i, \bar \alpha_{i} \rangle}a_{i+1}) s_{i+1} a_i^{ s_{i+1}^{-1}\bar\alpha_{i}^\vee},$$
where we again emphasize that the action of \(s_{i+1}^{-1}\) on \(\bar\alpha_{i}^\vee\) is via the \emph{finite} Weyl group action. 
Arguing inductively, we find that 
$$ \widetilde{\psi}_{\tld{z}}(a,t) \simeq \left(\dots, u_k\left( a_i^{\langle \bar \alpha_k, s_{k-1}^{-1}\dots s_{i+1}^{-1}\bar\alpha_{i}^\vee \rangle} a_k \right) s_k a_i^{ s_k^{-1} s_{k-1}^{-1}\dots s_{i+1}^{-1}\bar\alpha_{i}^\vee} t \right) ,$$
where the $i$'th coordinate for $i>1$ is contained in $\bar I_1 s_i \bar I_1$ (and its precise form is not of importance).
The conclusion now follows from Lemma \ref{function-of-t}.
\end{proof}
Finally, we can specialize the lemma above to the case when $f$ is a character of $T$.
More precisely, given a rational character $\beta : T \to \bG_m \subset \bA^1$ we can view $\beta$ as a regular function on $T$ and consider the resulting function $x_\beta : \tld{D}^\circ({\tld{z}}) \to \bA^1$ as in Definition \ref{function-on-demazure-resolution}.
\begin{prop}
\label{extending-functions-prop}
For any rational character $\beta \in X^*(T)$, the function $x_\beta : \tld{D}^\circ({\tld{z}}) \to \bA^1$ extends to $\tld{D}^\circ({\tld{z}},i)$ if and only if
$\langle \beta, s_k^{-1} s_{k-1}^{-1}\dots s_{i+1}^{-1}\bar\alpha_{i}^\vee \rangle \geq 0$, where the action of each \(s_{j}^{-1}\) is via the \emph{finite} Weyl group action. 
\end{prop}
\begin{proof}
Consider the open chart $\widetilde{\psi}_{\tld{z}} :  \bA^k\times T \hookrightarrow \tld{D}^\circ({\tld{z}},i)$, which covers the complement of $\tld{D}^\circ({\tld{z}})$ in $\tld{D}^\circ({\tld{z}},i)$.
Then $x_\beta$ extends to $\tld{D}^\circ({\tld{z}},i)$ if and only if the rational function $(\widetilde{\psi}_{\tld{z}})^* x_\beta : T \times \bA^k \to \bP^1$ is regular.
By Lemma \ref{function-in-other-chart}, we have
$$ (\widetilde{\psi}_{\tld{z}})^* x_\beta (a,t) = \beta\left(a_i^{ s_k^{-1} s_{k-1}^{-1}\dots s_{i+1}^{-1}\bar\alpha_{i}^\vee} t \right) =  a_i^{ \langle \beta , s_k^{-1} s_{k-1}^{-1}\dots s_{i+1}^{-1}\bar\alpha_{i}^\vee \rangle} \beta(t) . $$
The result is now immediate.
\end{proof}

We will now consider the problem of extending specific functions which will be relevant in our application in Section \ref{section:step3}.
\begin{thm} \label{important}
Assume that ${\tld{z}} C_0$ is antidominant, or more generally that ${\tld{z}} C_0 \uparrow C_0$ in Jantzen's ordering \cite[Section II.6.5]{jantzen}.
Then for all $j \in \lbrace 1,\dots,n \rbrace$, the function
\begin{align*}
x_j : \bar{I} {\tld{z}} \bar{I}_1 / \bar{I}_1 & \to \bA^1 \\
A {\tld{z}} \bar{I}_1 & \mapsto \text{upper left }(j \times j)\text{-minor of }A \mod v
\end{align*}
extends to a regular function on the closure $\overline{\bar{I} {\tld{z}} \bar{I}_1 / \bar{I}_1} \subset L \GL_n / \bar{I}_1$.
\end{thm}
\begin{proof}
By Lemma \ref{extend} , $x_j$ extends to a regular function on $S({\tld{z}})$ if $x_j$ extends across $D({\tld{z}},i)$ for all $i\in \{1,...,l({\tld{z}})\}$.
This is what we will show.

Write ${\tld{z}} = t_\eta z$, where $\eta \in \Phi^\vee$ and $z \in W$.
Recall the character \(\omega_{j} : (t_{1},\dots,t_{n}) \mapsto t_{1}t_{2}\cdots t_{j}\) of \(T\). 
Note that for $(t,A') \in T \times \bar I_1$, we have $A' {\tld{z}} t \bar{I}_1 = A' \Ad_{z}(t) {\tld{z}} \bar{I}_1$, which implies that $x_j(A' {\tld{z}} t \bar{I}_1) = \omega_j(\Ad_{z} t) = (z^{-1} \omega_j)(t)$, and therefore $$x_j = x_{z^{-1} \omega_j}.$$
By Proposition \ref{extending-functions-prop}, it follows that $x_j$, as a function on $\tld{D}^\circ({\tld{z}})$, extends to a regular function on $\tld{D}^\circ({\tld{z}},i)$ if and only if $$0 \leq \langle z^{-1} \omega_j, s_{k}^{-1} s_{k-1}^{-1} \cdots s_{i+1}^{-1}(\bar \alpha_i^\vee) \rangle = \langle \omega_j , s_1 \cdots s_i(\bar \alpha_i^\vee) \rangle . $$ 
In turn, this holds for all $j$ if and only if
\begin{equation}
\label{eq:not-so-mysterious-condition}
s_1 s_2 \cdots s_i(\bar \alpha_i^\vee) \in \bR_{\geq 0} \cdot \Phi^{\vee,+} \,\,\,\,\,\,\, \text{ for all }i \in \lbrace 1,\dots,k\rbrace .
\end{equation}
We again emphasize that the action of each \(s_{\bullet}\) here is via the \emph{finite} Weyl group action. 
In contrast, we will in the remainder of this proof denote the action of \(s \in W_{a}\) on a point \(y \in X_{*}(T)\) via the \emph{affine} Weyl group action by \(s \cdot_a y\).

We now turn to the proof of \eqref{eq:not-so-mysterious-condition}.
For each simple reflection $s$, let $v_s$ be the vertex of the dominant base alcove opposite to the wall which $s$ reflect across.
Then by construction, we have $s \cdot_a v_s - v_s = \bar \alpha_s^\vee$ for all simple reflections, or equivalently $v_s - s \cdot_a v_s = s(s \cdot_a v_{s} - v_{s}) = s(\bar \alpha_s^\vee)$ by applying the affine transformation $s$.
In particular, using this relation for $s = s_i$ and applying $s_1\cdots s_{i-1}$, we get
\begin{align*}
s_1  \cdots s_{i-1} \cdot_a v_{s_i} - s_1 \cdots s_{i-1} \cdot_{a} s_{i} \cdot_a v_{s_i} & = s_{1}\cdots s_{i-1}\left( v_{s_i} - s_{i} \cdot_a v_{s_i}\right) =  s_1 \cdots s_{i-1} s_{i}(\bar \alpha_i^\vee) .
\end{align*}
Note that we have used the affine Weyl group action on the left and the finite Weyl group action on the right hand side of this formula. 
It is clear that \(s_{1}  \cdots  s_{i-1} \cdot_a v_{s_i} - s_{1} \cdots s_{i} \cdot_a v_{s_i} \in \Phi^{\vee,+} \cup \lbrace 0 \rbrace\)
because the factorization \(s_{1}s_{2} \cdots s_{k} = {\tld{z}}\) corresponds to an alcove walk in the downwards direction. 
We conclude \eqref{eq:not-so-mysterious-condition}, so the function \(x_{j}\) extends for each \(j\) as claimed. 
\end{proof}

To be able to apply this theorem in Section \ref{section:step3}, we need to translate back our spaces to the set-up we used in there. To do this, define the involution
$$ *:LG\to LG$$
which is given by transposition of matrices. This induces an isomorphism
$$ LG/\bar I\cong I\setminus LG$$
and induces the $*$-involution on $\tld{W}$ introduced in Section \ref{presirred}.

\begin{cor}
\label{extending-minors-in-general}
Assume that ${\tld{z}}$ is as in Theorem \ref{important}.
Then the function
\begin{align*}
x_j : I_1 \backslash I_1 {\tld{z}}^* I & \to \bA^1 \\
I_1 {\tld{z}}^* A & \mapsto \text{upper left }(j \times j)\text{-minor of }A \mod v 
\end{align*}
extends to the Schubert variety $\overline{I_1 \backslash I_1 {\tld{z}}^* I} \subset \tld{\Fl}$.
\end{cor}

\begin{proof}
The $*$-involution induces isomorphisms fitting in the following commutative diagram
\[\begin{tikzcd}
	{I_1 \backslash I_1 {\tld{z}}^* I} & {\overline{I_1 \backslash I_1 {\tld{z}}^* I}} \\
	{\bar{I} {\tld{z}} \bar{I}_1 / \bar{I}_1} & {\overline{\bar{I} {\tld{z}} \bar{I}_1 / \bar{I}_1}.}
	\arrow["\cong"', from=1-1, to=2-1]
	\arrow["\cong"', from=1-2, to=2-2]
	\arrow[hook, from=1-1, to=1-2]
	\arrow[hook, from=2-1, to=2-2]
\end{tikzcd}\]
We can extend this commutative diagram
\[\begin{tikzcd}
	{I_1 \backslash I_1 {\tld{z}}^* I} & {\overline{I_1 \backslash I_1 {\tld{z}}^* I}} \\
	{\bar{I} {\tld{z}} \bar{I}_1 / \bar{I}_1} & {\overline{\bar{I} {\tld{z}} \bar{I}_1 / \bar{I}_1}} \\
	&& {\bA^1}
	\arrow["\cong"', from=1-1, to=2-1]
	\arrow["\cong"', from=1-2, to=2-2]
	\arrow[hook, from=1-1, to=1-2]
	\arrow[hook, from=2-1, to=2-2]
	\arrow["{x_j}"', from=2-1, to=3-3]
	\arrow["{\overline{x_j}}", from=2-2, to=3-3]
	\arrow[curve={height=-18pt}, dashed, from=1-2, to=3-3]
\end{tikzcd}\]
If we let \(x_{j}^*\) denote the function from Theorem \ref{important} to distinguish it from we now call \(x_{j}\), 
note that $x_j=x_j^*\circ *$, i.e. taking the upper left $(j\times j)$-minor is invariant under transposition. Therefore, if we can find a morphism for the dashed arrow, extending the commutative diagram, we are done. But by Theorem \ref{important}, we can extend the morphism $x_j^*$ to $\overline{x_j}$. Therefore we can define $\overline{x_j}\circ *:\overline{I_1 \backslash I_1 {\tld{z}}^* I}\to \bA^1 $ and this is desired morphism. 

\end{proof}
\end{appendices}

\newpage
\bibliography{Biblio}
\bibliographystyle{amsalpha}

\end{document}